\documentclass{amsart}[11pt]

\newtheorem{theorem}{Theorem}[section]
\newtheorem{lemma}[theorem]{Lemma}
\newtheorem{trditev}[theorem]{Proposition}
\newtheorem{posledica}[theorem]{Corollary}

\theoremstyle{definition}

\theoremstyle{remark}
\newtheorem{remark}[theorem]{Remark}

\numberwithin{equation}{section}

\usepackage{amssymb}
\usepackage{pstricks}
\usepackage{pst-node}

\def\max{\mathop{\rm max}\nolimits}
\def\grad{\mathop{\rm grad}\nolimits}
\def\dim{\mathop{\rm dim}\nolimits}
\def\supp{\mathop{\rm supp}\nolimits}
\def\Hom{\mathop{\rm Hom}\nolimits}

\def\rmax{\mathop{\rm rmax}\nolimits}

\newcommand{\C}{\mathbb{C}}
\newcommand\dibar{\overline{\partial}}

\parskip=\smallskipamount

\begin{document}
\title[Stein neighborhood basis]
{Stein neighborhood bases of embedded strongly pseudoconvex domains and approximation of mappings}
\author{Tadej Star\v{c}i\v{c}}
\address{Institute of Mathematics, Physics and Mechanics, University of Ljubljana, Jadranska 19, 1000 Ljubljana, Slovenia}
\email{tadej.starcic@fmf.uni-lj.si}
\date{April 25, 2008}


\keywords{Stein spaces, strongly pseudoconvex domains, fiber bundles, holomorphic mappings, approximation\\
\indent Research supported by grants ARRS (3311-03-831049), Republic of Slovenia.}

\begin{abstract}
In this paper we construct a Stein neighborhood basis for any compact subvariety $A$ with strongly pseudoconvex boundary $bA$ 
and Stein interior $A\backslash bA$ in a complex space $X$. This is an extension of a well known theorem of Siu. When $A$ is a complex curve, our result coincides
with the result proved by Drinovec-Drnov\v{s}ek and Forstneri\v{c}.
We shall adapt their proof to the higher dimensional case,
using also some ideas of Demailly's proof of Siu's theorem.
For embedded strongly pseudoconvex domain 
in a complex manifold we also find a basis of 
tubular Stein neighborhoods.
These results are applied to the approximation problem
for holomorphic mappings.
\end{abstract}
\maketitle

\section{Introduction}

One of the most important classes of complex manifolds is the class
of Stein manifolds, introduced in 1951 by K.\ Stein in his work
on the second Cousin problem (see \cite{Stein}).
Stein manifolds can be characterized as the closed complex submanifolds of 
complex Euclidean spaces $\mathbb{C}^N$ (see Remmert \cite{Remmert1}), 
by the existence of strictly plurisubharmonic exhaustion functions 
(see Grauert \cite{lit9}), and also by the vanishing of the cohomology groups  
with coefficients in coherent analytic sheaves (Cartan's Theorems A and B).
For the general theory of Stein manifolds and Stein spaces (with singularities)
see the monographs \cite{Grauert-Remmert,lit21,Ho}.

Since many classical problems are solvable on Stein spaces,
it is an enormously useful property for a subset of a complex space 
to have an open Stein neighborhood, or even a basis of such neighborhoods. 
A seminal result in this direction is the theorem of Y.-T.\ Siu 
to the effect that every Stein subvariety  
in a complex space admits an open Stein neighborhood (see \cite{lit3};
for different proofs and extensions see also \cite{Col, lit1,lit5}). 
In this paper we extend Siu's theorem to compact subvarieties with strongly pseudoconvex boundaries and Stein interior 
in complex spaces (see Theorem \ref{izrekg}). 
For embedded strongly pseudoconvex domains  
in complex manifold we also find bases of 
tubular Stein neighborhoods (see Theorem \ref{izrek2}).
These results are applied to the approximation problem
for holomorphic mappings (see Corollaries \ref{corollary1}
and \ref{CR-approx}).

All complex spaces in this paper will be assumed reduced and paracompact.
Let $l\in\{2,3,\ldots,\infty\}$.
A compact complex subvariety $A$ of a complex space $X$ has 
{\em embedded $\mathcal{C}^l$-boundary}, $bA$, if 
every point $p\in bA$ admits an open neighborhood 
$U\subset X$ and a holomorphic embedding 
$\Phi\colon U\hookrightarrow U'\subset\C^{N}$ onto a closed complex 
subvariety $\Phi(U)$ in a domain $U'\subset\C^{N}$
such that $\Phi(A\cap U)$ is a complex submanifold 
of $U'$ with $\mathcal{C}^l$-boundary $\Phi(bA\cap U)$. 
The interior $A\backslash bA$ is a closed complex subvariety
of $X\backslash bA$ that has no singularities near
$bA$. The boundary $bA$ is {\em strongly pseudoconvex} 
at a point $p\in bA$ if we can choose a neighborhood $U\ni p$ 
as above such that $\Phi(A\cap U)$ projects biholomorphically 
onto a piece of a strongly pseudoconvex domain in a 
suitable lower dimensional complex subspace of $\mathbb{C}^N$.

A compact set $K$ in a complex space $X$ is said to be
{\em $\mathcal{O}(X)$-convex}, if for every point $p\in X\backslash K$
there exists a holomorphic function $f$ on $X$ satisfying
$|f(p)|>\sup_{x\in K}|f(x)|$.

The following theorem is our first main result; we prove it in section \ref{Sb}.

\begin{theorem}
\label{izrekg}
Let $A$ be a compact subvariety with embedded strongly pseudoconvex 
$\mathcal{C}^2$-boundary $bA$ in a complex space $X$ 
such that the interior $A\backslash bA$ is a Stein space.
Assume that $K\subset \Omega$ is a compact $\mathcal{O}(\Omega)$-convex set 
in an open Stein set $\Omega\subset X$ such that 
$A\cap K$ is $\mathcal{O}(A)$-convex and $K\cap bA=\emptyset$. 
Then $A\cup K$ has a fundamental basis of open Stein neighborhoods in $X$.
\end{theorem}

\begin{figure}[ht]

\pspicture(-2.5,0.5)(8.5,4.5)
\psellipse[fillstyle=solid,fillcolor=gray](3,2.5)(1,0.5)
\psellipse[linecolor=gray,hatchcolor=lightgray,fillstyle=hlines](3,2.5)(1.5,1.8)
\psline[linewidth=1pt,linearc=.55]{*-*}(5.5,2.5)(5,3)(4.5,2.8)(4,2.6)(3.5,2.4)(3,2.5)(2.5,2.7)(2,2.5)(1.5,2.4)(1,2.3)(0.5,2.5)
\psline[linewidth=1pt,linearc=.55]{-}(5.5,2.7)(5,3.2)(4.5,3)(4,2.8)(3.5,3.1)(3,3.2)(2.5,3.1)(2,2.9)(1.5,2.6)(1,2.7)(0.5,2.7)
\psline[linewidth=1pt,linearc=.55]{-}(0.5,2.3)(1,2.2)(1.5,2.1)(2,2)(2.5,1.8)(3,1.8)(3.5,1.9)(4,2.2)(4.5,2.6)(5,2.6)(5.3,2.3)(5.5,2.3)
\psline[linewidth=1pt,linearc=.15]{-}(0.5,2.7)(0,2.5)(0.5,2.3)
\psline[linewidth=1pt,linearc=.15]{-}(5.5,2.7)(6,2.5)(5.5,2.3)
\rput(3,3.5){$\Omega$}
\rput(3.3,2.7){$K$}
\rput(1,3.3){\rnode{a}{$A$}}
\rput(1,2.25){\rnode{b}{ }}
\ncline[nodesep=3pt, linewidth=0.3pt]{->}{a}{b}
\rput(0.7,1.6){\rnode{e}{$bA$}}
\rput(5.9,3.2){\rnode{f}{$bA$}}
\rput(0.5,2.55){\rnode{g}{}}
\rput(5.5,2.4){\rnode{h}{}}
\ncline[nodesep=3pt, linewidth=0.3pt]{->}{e}{g}
\ncline[nodesep=3pt, linewidth=0.3pt]{->}{f}{h}
\rput(5,3){\rnode{c}{}}
\rput(5,2.0){\rnode{d}{$A$}}
\ncline[nodesep=3pt, linewidth=0.3pt]{<-}{c}{d}
\endpspicture
\caption{Theorem \ref{izrekg}}
\label{slika1}
\end{figure}

Note that the subvariety $A$ in Theorem \ref{izrekg} (see Figure \ref{slika1}) has at most finitely many
singularities and at most finitely many irreducible components,
possibly of different dimensions.
When $A$ is a complex curve ($\dim A=1$), Theorem \ref{izrekg} coincides
with Theorem 2.1 in the paper \cite{lit2} by Drinovec-Drnov\v{s}ek and Forstneri\v{c}.
We shall adapt the proof given there to the higher dimensional case,
using also some ideas of Demailly's proof (see \cite{lit1}) of Siu's theorem. 
For Stein neighborhoods of complex curves in $\C^n$ see also 
Wermer \cite{lit7} and Stolzenberg \cite{Stolz}.

It would be interesting to generalize Theorem \ref{izrekg} 
to the case when $A$ has smooth {\em weakly pseudoconvex}
boundary $bA$. In this case the problem of the existence of a Stein
neighborhood basis  is quite delicate
already when $A$ is the closure of a smoothly bounded 
weakly pseudoconvex domain in a Stein manifold.
The famous {\em worm domain} of Diederich and Forn\ae ss
(see \cite{DF77a}) shows that the answer may be negative in general
and additional hypotheses are needed. For results in this direction
see the papers by Diederich and Forn\ae ss \cite{DF77b}, 
Bedford and Forn\ae ss \cite{BF78}, Sibony \cite{Sib87a,Sib91},
Stens\o nes \cite{Ste87}, Forstneri\v c and Laurent-Thi\'ebaut \cite{FLaurent}
and others. Virtually nothing seems known about the existence of
Stein neighbohoods of embedded domains of this type in higher dimensional
complex manifolds.

When solving complex analytic problems on, or near a submanifold of 
a complex manifold, it is often useful to have 
{\em tubular} Stein neighborhoods; these are open Stein neighborhoods 
that are biholomorphically equivalent to neighborhoods of the zero section in a 
holomorphic vector bundle (for the precise definition see section \ref{TS}). 
We prove that a relatively compact strongly pseudoconvex 
domain that is holomorphically embedded with $\mathcal{C}^2$-boundary 
in an arbitrary complex manifold admits a basis of such neighborhoods. 
To be precise, let $D$ denote a relatively compact strongly pseudoconvex domain 
with  $\mathcal{C}^2$-boundary in a Stein manifold $S$; this means that there 
exists a $\mathcal{C}^2$ strictly plurisubharmonic function $\rho \colon U \to \mathbb{R}$ 
in an open set $U\subset S$ containing $\overline D$ such that 
$D=\{z \in U \mid \rho(z)<0\}$ and $d\rho(z)\neq 0$ for every $z \in bD=\{\rho =0\}$.

The following result is proved in section \ref{TS}
(see Theorem \ref{izrek3}).

\begin{theorem}\label{izrek2}
Let $D$ be a relatively compact strongly pseudoconvex domain with $\mathcal{C}^2$-boundary
in a Stein manifold $S$, and let $X$ be a complex manifold. 
Assume that $f\colon\overline{D}\hookrightarrow X$ is a $\mathcal{C}^2$-embedding that is holomorphic in $D$. 
Then there exists a holomorphic vector bundle $\pi \colon \theta \to U_{\overline{D}}$ 
over an open neighborhood $U_{\overline{D}}$ of $\overline{D}$ in $S$ 
such that $A=f(\overline{D})$ has a basis of open Stein neighborhoods in $X$ 
that are biholomorphic to neighborhoods (in $\theta$) of the zero section 
of the restricted bundle $\theta|_{\overline{D}}$. 
\end{theorem}

Theorem \ref{izrek2} is known in two special cases:
When $S$ is an open Riemann surface (see \cite[Theorem 1.7]{lit2}),
and when $X$ is the total space of a holomorphic submersion
$h\colon X \to S$ and $f\colon \overline{D}\to X$ is a 
continuous section of $h$ that is holomorphic in $D$
(see \cite[Theorem 1.1]{lit25}). In the latter result, which is
proved by the method of holomorphic sprays, merely the continuity
of $f$ up to the boundary suffices. In the general case considered here,
the stronger hypothesis in our Theorem \ref{izrek2} 
is justified by the following example (see \cite[Example 5.4]{lit25}):

\smallskip
\textit{There exists a smooth injective map $f$ from 
the closed unit ball $\overline {\,\mathbb{B}}$ in $\C^5$ into $\C^{8}$, 
that is a holomorphic embedding of the open unit ball $\mathbb B$, such that 
$f(\overline {\,\mathbb{B}})$ has no basis of Stein neighborhoods.}
\smallskip

The bundle $\theta$ in Theorem \ref{izrek2} is essentially the normal bundle 
of $A$ in $X$. If this bundle is trivial, the proof for the Riemann surface case 
given in \cite{lit2} applies without essential changes and gives a basis of
open Stein neighborhoods of $\overline{D}$ that are biholomorphic to domains 
in $S\times \mathbb{C}^m$, where $m=\dim X-\dim S$. 
If on the other hand the normal bundle of $A$ in $X$ is nontrivial, 
we employ a somewhat different construction 
(see Proposition \ref{trditev22} below).

Theorem \ref{izrek2} enables us to prove the following approximation result.

\begin{posledica}
\label{corollary1}
Let $X$ and $Y$ be complex manifolds, and let $A\subset X$ be
as in Theorem \ref{izrek2}. Then every continuous map 
$g\colon A\to Y$ that is holomorphic in the interior 
$A \backslash bA$ of $A$ can be approximated, uniformly 
on $A$, by holomorphic maps from open neighborhoods of $A$ to $Y$. 
\end{posledica}

Corollary \ref{corollary1} is a special case of Theorem \ref{CR-approx}
in section \ref{approximation}. In general the domains of the approximating maps in 
Corollary \ref{corollary1} shrink down to $A$; only for a certain
special class of target manifolds $Y$ (i.e., for those
satisfying the Oka property, see \cite{lit5}) one obtains 
a sequence of approximating maps from a fixed neighborhood 
of $A$ to $Y$.

\section{Stein neighborhood basis}\label{Sb}

In this section we prove Theorem \ref{izrekg}.

We begin with  preparatory material. 
Given a $\mathcal{C}^2$-function $u$ on a complex manifold $X$, 
we define a Hermitean metric on the complexified tangent bundle 
$\mathbb{C}TX =\mathbb{C}\otimes_{\mathbb{R}} TX$ 
in a system of local holomorphic coordinates $z=(z_1,\ldots,z_n)$ 
on $X$ by
$$H(u)_z=\sum_{j,k=1}^n\frac{\partial^2 u}{\partial z_j\partial\bar{z}_k }(z)\, \textrm{d}z_j\otimes\textrm{d}\bar{z}_k.$$
The definition is independent of the choice of holomorphic coordinates. 
The {\em Levi form} of $u$ is given by
\[
	\mathcal{L}_u(z;\xi)= \left\langle H(u)_z, \xi \otimes \bar \xi\right\rangle =
  \left\langle \partial\overline{\partial} u(z), \xi \wedge \bar \xi\, \right\rangle, \quad \xi \in T_z^{1,0}X,
\]
where $T_z^{1,0}X$ is the eigenspace corresponding to the eigenvalue $i$ 
of the underlying almost complex structure operator $J$ on $\mathbb{C}TX$. 
In local coordinates we have 
\[
   \mathcal{L}_u(z;\xi)= \sum_{j,k =1}^n \frac{\partial^2 u}{\partial z_j\partial\bar{z}_k }(z)
\, \xi_j\overline{\xi}_k, \quad \xi=\sum_{j=1}^n \xi_j \frac{\partial}{\partial z_j}\in T_z^{1,0}X.
\]
A function $u$ is {\em strictly plurisubharmonic} if and only if $H(u)$ is a 
positive definite Hermitean metric on $X$, and this is the case if and only if
$\mathcal{L}_u$ is a positive definite Hermitean quadratic form.

Recall that a function $u$ on a complex space $X$ (with singularities) is said to be strictly plurisubharmonic, if there exist a cover of $X$ by open sets $\{U_{\lambda}\}_{\lambda \in I}$ and holomorphic embeddings 
$Z_{\lambda}\colon U_{\lambda} \to U_{\lambda}'\subset \mathbb{C}^N, \lambda \in I$,
where $Z_{\lambda}(U_{\lambda})$ is a closed complex subvariety in a domain $U_{\lambda}'\subset \mathbb{C}^N$,
such that $(u \circ Z_{\lambda}^{-1})|_{Z_{\lambda}(U_{\lambda})}$ admits an extension 
to $U_{\lambda}'$ that is strictly plurisubharmonic. 
The notion of strict plurisubharmonity  does not depend on 
the choice of the covering, nor on the embeddings $Z_{\lambda}, \lambda \in I$.

Let $A$ be as in Theorem \ref{izrekg}. For simplicity we shall assume that $A$ has pure dimension $n$,
although the same proof will also apply in the case of irreducible
components of different dimensions.

The  conditions on $A$ imply that for every point $p\in bA$ there exist 
an open Stein neighborhood $U_p\subset X$ of $p$, with $U_p \cap K=\emptyset$, 
and a holomorphic embedding $Z_p=(z,w)\colon U_p\to\mathbb{C}^{n +n_p}$ 
such that $Z_p(U_p)$ is a closed complex subvariety of the polydisc 
\[
   U_p '=\{(z,w)\in \mathbb{C}^{n +n_p} \mid |z_1|, \ldots , |z_n|, |w_1|, \ldots , |w_{n_p}|<1 \}=\Delta^{n+n_p},
\]   
where $z=(z_1,\ldots, z_n)\in\C^n$ and $w=(w_1,\ldots w_{n_p})\in\C^{n_p}$, such that
\begin{equation}\label{L1}
Z_p(A\cap U_p)=\{(z,w)\in U_p'\mid z\in \Gamma_p, \ w=g_p(z)\}.
\end{equation}
Here we denoted by
\begin{equation}\label{L3}
\Gamma_p=\{z\in \mathbb{C}^n \mid |z_1|, \ldots , |z_n|<1, h_p(z)\leq 0\}
\end{equation}
a connected convex set with the interior 
\begin{equation}\label{L4}
\Gamma_p\backslash b\Gamma_p=\{z\in \mathbb{C}^n \mid |z_1|, \ldots , |z_n|<1, h_p(z)<0\},
\end{equation}
where $h_p$ is a $\mathcal{C}^2$ strictly convex real function (its real Hessian is positive), with d$h_j \neq 0$,
and $g_p\colon \Delta^{n} \to\Delta^{n_p}$ is a $\mathcal{C}^2$-map that 
is holomorphic in the interior of $\Gamma_p$ (see Figure \ref{slika2}).

\begin{figure}
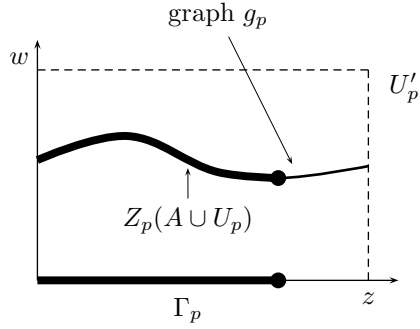

\psset{unit=0.8 cm}

\pspicture(-3,-1)(8.5,5)
\psline[linewidth=0.5pt,linearc=.15]{->}(0,0)(6,0)
\psline[linewidth=0.5pt,linearc=.15]{->}(0,0)(0,4)
\psline[linewidth=3pt,linearc=.15]{-}(0,0)(4,0)
\pscurve[linewidth=1pt,showpoints=false]{-}(0,2)(1.5,2.4)(3,1.8)(4,1.7)(5.5,1.9)
\pscurve[linewidth=3pt,showpoints=false]{-}(0,2)(1.5,2.4)(3,1.8)(4,1.7)
\psline[linewidth=0.5pt,linestyle=dashed,dash=3pt 2pt,linearc=.15]{-}(5.5,0)(5.5,3.5)
\psline[linewidth=0.5pt,linestyle=dashed,dash=3pt 2pt,linearc=.15]{-}(0,3.5)(5.5,3.5)
\psdots[linewidth=2pt](4,1.7)(4,0)

\rput(2.5,1){\rnode{g}{$Z_p(A\cup U_p)$}}
\rput(2.5,2){\rnode{h}{}}
\ncline[nodesep=3pt, linewidth=0.3pt]{->}{g}{h}
\rput(3,4.4){\rnode{i}{graph $g_p$}}
\rput(4.3,1.8){\rnode{j}{}}
\ncline[nodesep=3pt, linewidth=0.3pt]{->}{i}{j}
\rput(-0.3,3.7){$w$}
\rput(5.5,-0.3){$z$}
\rput(2.5,-0.5){$\Gamma_p$}
\rput(6.1,3.2){$U_p '$}
\endpspicture
\caption{The sets $\Gamma_p$ and $Z_p(A\cap U_p)$}
\label{slika2}
\end{figure}

We choose smaller open sets $V_p$ and $V_p'$ such that $V_p \Subset V_p'\Subset U_p$ with $p \in V_p$. Since $bA$ is compact, we can cover $bA$ with finitely many sets $V_{1},\ldots , V_{m}$ described above.
Let $Z_j$, $g_j$ and $h_j$ be the corresponding maps, and let 
$U_j$, $V_j'$ $U_j'$ and $\Gamma_j$ be the corresponding sets for all 
$j\in\{1,\ldots ,m\}$. For every $j\in \{1,\ldots ,m\}$ we define
\begin{eqnarray}
\label{phij}
	\phi_j(x) &=& w(x)-g_j\bigl(z(x)\bigr), \quad x \in U_j,
\end{eqnarray}
\begin{eqnarray}	
\label{lamj}	
	\Lambda_j &=& \{x\in U_{j}\colon z(x)\in \Gamma_j\}, \\
\label{lamj'}
	\Lambda_j' &=& (V_j\cap\Lambda_j) \backslash \bigcup_{i=1, i \neq j}^m(V_i \backslash \Lambda_i).
\end{eqnarray}
The map $\phi_j$ and the sets $\Lambda_j$ and $\Lambda_j'$ describe 
the subvariety $A \subset X$ locally in a neighborhood of its boundary $bA\cap U_j$ 
(see Figure \ref{slika3}).

\begin{figure}[ht]
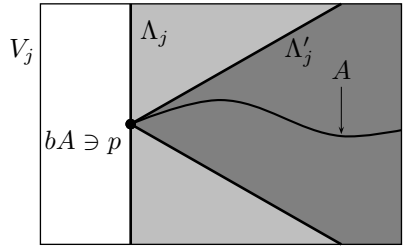

\psset{unit=0.8 cm}
\pspicture(2,-0.5)(10,4.5)
\pspolygon*[linecolor=lightgray](4.5,4)(9,4)(9,0)(4.5,0)
\pspolygon*[linecolor=gray](8,0)(9,0)(9,4)(8,4)(4.5,2)
\pspolygon[linewidth=0.5pt](3,0)(9,0)(9,4)(3,4)
\psline[linewidth=1pt,linearc=.15]{-}(4.5,0)(4.5,4)
\psline[linewidth=1pt,linearc=.15]{*-*}(4.5,2)(4.5,2)
\psline[linewidth=1pt,linearc=.15]{-}(4.5,2)(8,4)
\psline[linewidth=1pt,linearc=.15]{-}(4.5,2)(8,0)
\pscurve[showpoints=false]{-}(4.5,2)(6,2.4)(8,1.8)(9,1.9)
\rput(3.7,1.7){$bA\ni p$}
\rput(4.9,3.5){$\Lambda_j$}
\rput(7.3,3.2){$\Lambda_j '$}
\rput(8,2.9){\rnode{g}{$A$}}
\rput(8,1.7){\rnode{h}{}}
\ncline[nodesep=3pt, linewidth=0.3pt]{->}{g}{h}
\rput(2.7,3.2){$V_j$}
\endpspicture
\caption{The sets $\Lambda_j$ and $\Lambda_j'$}
\label{slika3}
\end{figure}

When $\dim A=1$, each connected component of $bA$ 
is a closed Jordan curve of class $\mathcal{C}^2$ which 
admits an open Stein neighborhood (see \cite[Lemma 2.2]{lit2}), 
and in this case we can choose $\Lambda_j$ and $\phi_j$ globally around 
the boundary $bA$. When $\dim A>1$, such Stein neighborhoods clearly
do not exist, and hence this can be done only locally 
as described above.

Let us choose additional open sets $V_{m+1},\ldots ,V_N$ and slightly larger open sets $U_{m+1},\ldots ,U_N$ ($V_i \Subset U_i$) in $X$ whose closures do not intersect any of the sets $U_j \backslash \Lambda_j'$ 
for every $j \in \{1,\ldots m\}$, and such that 
\[
   A \cup K \subset \bigcup_{i=1} ^N V_j.
\]
By choosing the sets $V_j\subset U_j$ sufficiently small, we get holomorphic maps 
$\phi_j\colon U_j\to\mathbb{C}^{n_j}$ whose components generate the ideal sheaf of 
$A$ at every point of $U_j$. 
(If $U_j\cap A=\emptyset$ for some $j\in\{ m+1,\ldots , N\}$, 
we take $n_j=1$ and $\phi_j(x)=1$ for every $x\in U_j$.)
For consistency of notation we set $\Lambda_j=U_j$ for $j = m+1,\ldots , N$.

The following is \cite[Lemma 2.5]{lit2}; we recall the proof since some of the 
settings will be used later on.

\begin{trditev}\label{trditev}
Given a compact complex subvariety $A\subset X$ with embedded 
$\mathcal{C}^2$ strongly pseudoconvex boundary $bA$ and Stein interior $A\backslash bA$,
there exists a $\mathcal{C}^2$-function $\rho_2 \colon X\to\mathbb{R}$ that is 
strictly plurisubharmonic in an open neighborhood of $A$ in $X$.
\end{trditev}

\begin{proof}
The conditions on $A$ imply that there exists a strictly plurisubharmonic
bounded $\mathcal{C}^2$ exhaustion function
$\rho_1\colon A\to (-\infty,0]$ with $\rho_1|_{bA}=0$.

For $j\in\{1,\ldots,m\}$ let $Z_j\colon U_j\to U_j'$, $\Gamma_j$, $\Lambda_j$ and $\phi_j$ be as above.
Denote by $\psi_j'\colon \Gamma_j\times\mathbb{C}^{n_j}\to\mathbb{R}$ the unique function that is independent of the second variable $w\in\mathbb{C}^{n_j}$ and satisfies $\rho_1=\psi_j'\circ Z_j$ on $A\cap U_j$. 
We can extend $\psi_j'$ to a $\mathcal{C}^2$-function $\psi_j'\colon U_j'\to\mathbb{R}$ 
that is independent of the variable $w$ and set 
\begin{equation} \label{izraz1}
\psi_j=\psi_j'\circ Z_j\colon U_j\to\mathbb{R}.	
\end{equation}
Since $\psi_j|_{A\cap U_j}=\rho_1$ which is strongly plurisubharmonic, 
there is an open set $\widetilde{\Gamma}_j\subset \mathbb{C}^n$ 
containing $\Gamma_j$ such that $\psi_j$ is plurisubharmonic in the 
open set 
\begin{equation}\label{wU}
\widetilde{U}_j=\{x\in U_j\mid z(x)\in\widetilde{\Gamma}_j\}.
\end{equation}

For $j\in\{m+1,\ldots,N\}$ we get in a similar way a strictly plurisubharmonic 
$\mathcal{C}^2$-function $\psi_j\colon U_j \to\mathbb{R}$ extending 
$\rho_1|_{A\cap U_j}$. (For the purpose of this proof we need not consider
the sets $U_j$ that do not intersect $A$.)

We now use an argument from the proof of \cite[Theorem 4]{lit1}. 
Let $\{\vartheta_j\}_{j=1}^N$ be a smooth partition of unity in a neighborhood of $A$ in $X$ with $\supp\vartheta_j\subset U_j$ for every $j$. We fix a number $\epsilon>0$ and set 
\[
  \rho_2(x)=\sum_{j=1}^{N} \vartheta_j(x) \bigl(
  \psi_j(x)+\epsilon^3\log\bigl(1+\epsilon^{-4}|\phi_j(x)|^2 \bigr) \bigr).
\]
It can be verified that $\rho_2$ is strictly plurisubharmonic in a neighborhood of 
$A$ in $X$ provided that $\epsilon>0$ is chosen small enough. 
Indeed, since the function $\phi_j$ vanishes along $U_j \cap A$, we have 
\[
  i\partial\bar{\partial}\epsilon^3\log\bigl(1+\epsilon^{-4}|\phi_j|^2\bigr)=\epsilon^{-1}i\partial\bar{\partial}|\phi_j|^2
\]
on $U_j \cap A_{\textrm{reg}}$. It follows that the Levi form of $\rho_2$ is positive definite on $A$ in the directions normal to $A$. Since for $x \in A$ we have $\rho_2(x)=\sum_j\vartheta_j(x)\psi_j(x)=\rho_1(x)$, the Levi form of $\rho_2$ on $A$ is positive definite also in the directions tangent to $A$. 
\end{proof}

For the sake of completeness we also recall 
\cite[Lemma 5.18.]{lit4} which will be used in 
the proof of Proposition \ref{trditev2} below.

\begin{lemma}\label{lemastrogopsh}
Let $\theta\colon\mathbb{R}\to \mathbb{R}_+$ be a nonnegative 
smooth function with support in $[-1,1]$ and 
such that $\int_{\mathbb{R}}\theta(s)\textrm{d}s=1$ 
and $\int_{\mathbb{R}}s\theta(s)\textrm{d}s=0$.
For an arbitrary $\eta=(\eta_1,\ldots,\eta_p)\in (0,\infty)^p$, $p \in \mathbb{N}$, 
the regularized maximum function 
\[
  \rmax_{\eta}(t_1,\ldots ,t_p)=\int_{\mathbb{R}^n}\max (t_1+s_1,\ldots,t_p+s_p)
  \prod_{1\leq j\leq    n}\theta\left(\frac{s_j}{\eta_j}\right)
  \textrm{d}s_1\ldots\textrm{d}s_p
\]
satisfies the following properties:
\begin{enumerate}
\item \label{lemai} $\rmax_\eta(t_1,\ldots ,t_p)$ is nondecreasing in all variables, smooth and convex on $\mathbb{R}^n$,
\item \label{lemaii} $\max (t_1,\ldots ,t_p) \leq \rmax_\eta(t_1,\ldots ,t_p) \leq\max (t_1+\eta_1,\ldots ,t_p+\eta_p)$,
\item \label{lemmaiii} if $t_j+\eta_j \leq\max_{k\neq j} (t_k-\eta_k)$ is satisfied, then \\
$\rmax_\eta(t_1,\ldots ,t_p)=\rmax_{(\eta_1,\ldots ,\hat{\eta}_j,\ldots ,\eta_p)}(t_1,\ldots ,\hat{t}_j,\ldots ,t_p)$,
\item \label{lemaiv} $\rmax_{\eta}(t_1+a,\ldots ,t_p+a)=\rmax_{\eta}(t_1,\ldots ,t_p)+a$, for all $ a\in\mathbb{R}$,
\item \label{lemav} if $u_1,\ldots ,u_p$ are plurisubharmonic on $X$ and satisfy $\mathcal{L}_{u_j}(z;\xi)\geq \gamma_z(\xi)$
for a continuous Hermitean form $\gamma_z$ on $T_z^{1,0}X$, 
then $u=\rmax_{\eta}(u_1,\ldots ,u_p)$ is plurisubharmonic 
and satisfies $\mathcal{L}_{u}(z;\xi)\geq \gamma_z(\xi)$.
\end{enumerate}
\end{lemma}

Note that the regularized maximum function with $\eta=(\eta_1,\ldots,\eta_p)$ can be 
chosen as close as we desire to the usual maximum as all $\eta_j$ $(1\leq j \leq p)$ approach to $0$.

\begin{proof}
The change of variables $s_j\mapsto s_j-t_j$ shows that $\rmax_{\eta}$ is smooth. All properties are immediate consequences of the definitions, except perhaps (\ref{lemav}). Since $\rmax_{\eta}$ is nondecreasing and convex, $\rmax_{\eta}(u_1,\ldots , u_p)$ is plurisubharmonic. Fix a point $z_0 \in X$ and $\epsilon >0$. All functions $$u_j'(z)=u_j(z)-\gamma_{z_0}(z-z_0)+\epsilon |z-z_0|^2$$ are plurisubharmonic near $z_0$ and it follows that 
\[
\rmax_{\eta}\bigl(u_1'(z),\ldots , u_p'(z)\bigr)=u(z)-\gamma_{z_0}(z-z_0)+\epsilon |z-z_0|^2
\]
is also plurisubharmonic near $z_0$. Since $\epsilon>0$ was arbitrary, (\ref{lemav}) follows. 
\end{proof}

We notice that the estimate in the conclusion of (\ref{lemav}) in Lemma \ref{lemastrogopsh} holds also if $u_1,\ldots,u_p$ are not necessary plurisubharmic, but only satisfy the given estimates. We shall use this fact later on in Proposition \ref{trditev2} for functions whose Levi forms have bounded negative part.

We now recall \cite[Lemma 2.4]{lit2} and sketch its proof,
referring to  \cite{lit2,lit5} for the details.

\begin{lemma}\label{lemaspmeja}
Let $U\subset X$ be an open set containing $A\cup K$, where $A$ and $K$ 
are as in Theorem \ref{izrekg}. 
Then there exist a neighborhood $W$ of $A\cup K$ in $X$ with $\overline{W}\subset U$ and a strictly plurisubharmonic $\mathcal{C}^2$-function $\rho\colon \overline{W}\to\mathbb{R}$ such that $\rho>0 $ on $bW$ and $\rho<0$ on $K$.
\end{lemma}

\begin{proof}
By using the $\mathcal{O}$-convexity properties of $K$ and $A\cap K$ we can construct
smooth functions $\rho_1,\widetilde{\rho}_0\colon X\to \mathbb{R}$ with the properties listed below. The restrictions  $\widetilde{\rho}_0|_A$ and $\rho_1|_A$ are strictly plurisubharmonic bounded exhaustions and $\rho_1$ is strictly plurisubharmonic on a compact neighborhood $K'$ of $K$ in $U\cap \Omega$, where $K'\cap A\subset A\backslash bA$. We also have 
\[
\widetilde{\rho}_0=\rho_1 \textrm{ on } \overline{\Omega}_1,\quad \widetilde{\rho}_0<0 \textrm{ on } K,\quad \widetilde{\rho}_0>0 \textrm{ on } \Omega \backslash \Omega_1,
\]
\[\quad \rho_1\geq\widetilde{\rho}_0>0 \textrm{ on } A\backslash \overline{\Omega}_1,\quad \rho_1>\widetilde{\rho}+1\textrm{ on }A\backslash \overline{\Omega}_{2},\quad \rho_1|_{bA}=c_1>2
\]
for some constant $c_1>2$ and open sets $\Omega_1$ and $\Omega_2$ in $\Omega$, and such that $K \subset \Omega_1 \subset \overline{\Omega}_1 \subset \Omega_2\subset K'$.

We use Proposition \ref{trditev} to get a $\mathcal{C}^2$ strictly plurisubharmonic function $\rho_2$ in a neighborhood of $A$ and such that $\rho_1|_A=\rho_2|_A$.
Next we set $$\rho=\rmax (\widetilde{\rho}_0,\rho_2-1),$$ where $\rmax$ is a regularized maximum (see Lemma \ref{lemastrogopsh}). We observe that there is a compact neighborhood $\overline{W}\subset U$ of the set
$A \cup \overline{\Omega}_1$ such that before running out of the domain of one of these functions inside $\overline{W}$,
the second one is larger and hence takes over. It follows that $\rho$ is well defined and smooth on $\overline{W}$. 
As $\rho=\widetilde{\rho_0}$ on $\Omega_1$, we have $\rho<0$ on $K$. Furthermore,  since $\widetilde{\rho}_0>0$ on $\Omega \backslash \Omega_1$ and $\rho_1\geq\widetilde{\rho}_0>0$ on $A\backslash \overline{\Omega}_1$, we see that after shrinking $W$ around $A \cup \overline{\Omega}_1$ we get $\rho>0$ on $bW$. 
\end{proof}

Choose an open set $V\subset X$ with $A\cup K\subset V\Subset \bigcup_{j=1} ^N V_j$ and set
\begin{eqnarray}\label{lam}
	\Lambda=\bigcup_{j=m+1} ^{N}(V_j \cap \overline{V})\cup \bigcup_{j=1} ^{m}(\Lambda_j' \cap \overline{V}).
\end{eqnarray}
Recall that the maps $\phi_j$ defined in the beginning of this section are holomorphic on $\Lambda_j$ for every $j \in \{1,\ldots, N\}$, hence $\phi_j$ is holomorphic on $V_j \cap \Lambda$ for all $j$.

As in \cite{lit2} we choose smooth functions $\tau_j\colon V_j\to\mathbb{R}$ which tend to $-\infty$ at $bV_j$. For every $\delta\in(0,1]$ we set
\begin{eqnarray*}
  v_{\delta,j}(x) &=& \log\bigl(\delta+|\phi_j(x)|^2\bigr)+\tau_j(x), \quad x\in V_j, \\
  v_\delta(x) &=& \rmax_{\eta} (\ldots, v_{\delta,j}(x),\ldots),
\end{eqnarray*}
where the regularized maximum (see Lemma \ref{lemastrogopsh}) is taken over all indices $j\in\{1,\ldots , N\}$ for which $x\in V_j$.

The following proposition is essential in the proof of Theorem \ref{izrekg}.

\begin{trditev} \label{trditev2}
The function $v_{\delta}\colon V_{\delta}\to\mathbb{R}$ $(\delta\in(0,1])$ defined as above is of class $\mathcal{C}^2$ on an open neighborhood $V_{\delta}$ of the set $\Lambda$. Furthermore, $v_{0}(x)=\lim_{\delta\to 0}v_{\delta}(x)$ is a continuous function on $\Lambda \backslash A$, it is of class $\mathcal{C}^2$ in the interior of $V\backslash \Lambda$,  and it satisfies $\{v_0=-\infty\}=A$. There exists a constant $M>-\infty$ such that 
\[
	i\partial\bar{\partial}v_{\delta}(x) > M,\qquad x\in \Lambda,\ 0<\delta\le 1.
\]
\end{trditev}

\begin{proof}
We adapt the proof \cite[Lemma 2.3]{lit2} (which uses ideas from \cite{lit1})
to our situation. Let $\phi_j$ be the functions defined at the beginning of this section
(see especially (\ref{phij})). We begin by proving that the quotients 
$\frac{|\phi_k|}{|\phi_j|}$, for $j,k\in\{1,\ldots, N\}$, are bounded on the intersection of their domains 
near the subvariety $A$.

We first consider these quotients in the interior of $A$ away from $bA$. 
The generators $\phi_j$ and $\phi_k$ for the ideal sheaf of $A$ can be expressed in terms of one another in the interior of the set $\Lambda_j \cap \Lambda_k$, and hence $\frac{|\phi_k|}{|\phi_j|}$ is bounded in any relatively compact subset in the interior of $\Lambda_j \cap \Lambda_k$. The quotient $\frac{|\phi_k|}{|\phi_j|}$ is then uniformly bounded away from $bA$ on the set $(\Lambda_j \cap \overline{V}_j)\cap (\Lambda_k \cap \overline{V}_k)$. Remember also that $\Lambda_k \cap \overline{V}_k=\overline{V}_k$ for $k \in \{m+1,\ldots,N\}$.

Next we consider the expression $\frac{|\phi_k|^2}{|\phi_j|^2}$ for $j,k\in\{1,\ldots,m\}$ near the boundary $bA$ of the subvariety $A$. 
Recall that for any $j\in\{1,\ldots,m\}$ the map $\phi_j\colon U_j\to \Delta^{n_j}$ is defined by $\phi_j(x)=w(x)-g_j(z(x))$, where $g_j=(g_j^1,\ldots,g_j^{n_j})\colon\Delta^{n}\to\Delta^{n_j}$ is a $\mathcal{C}^2$-map that is holomorphic in the interior of $\Gamma_j$, and $Z_j=(z,w)\colon U_j\to U_j'\subset \mathbb{C}^{n +n_j}$ is a holomorphic embedding.
Set $\widetilde{\phi}_j(z,w)=w-g_j(z)$; then  
\[
   \bigl|\widetilde{\phi}_j(z,w)\bigr|^2=\sum_{p=1}^{n_j}\bigl(w_p-g_j^p(z)\bigr)\bigl(\overline{w}_p-\overline{g}_j^p(z)\bigr).
\]
Further we write $z_l=y_l+iy_{l+n}$ for $l\in \{1,\ldots,n\}$ and $w_p=t_p+it_{p+n_j}$ for $p\in \{1,\ldots,n_j\}$, where $y=(y_1, \ldots, y_{2n})$ and $t=(t_1,\ldots,t_{2n_j})$ are real coordinates on $\mathbb{C}^n \approx \mathbb{R}^{2n}$, respectively on $\mathbb{C}^{n_j}\approx \mathbb{R}^{2n_j}$, and $g_j^p=u_j^p+iu_j^{p+n_j}$ is the sum of real and imaginary part of $g_j^p$. We write down the Taylor expansion of $|\widetilde{\phi}_j|^2$ around a point $(y_0,t_0)\in U_j'$ such that $t_0=g_j(y_0)$:
\begin{eqnarray*}
	\bigl|\widetilde{\phi}_j(y_0+y,t_0+t)\bigr|^2 &=&
	\sum_{l=1}^{2n}\sum_{p=1}^{n_j}
				\biggl( \Bigl(\frac{\partial u_j^{p}}{\partial y_l}(y_0,t_0)\Bigr)^2 + 
       \Bigl( \frac{\partial u_j^{p+n_j}}{\partial y_{l}}(y_0,t_0)\Bigr)^2 \biggr)y_l^2 \\
       && + \sum_{p=1}^{2n_j}t_{p}^2 
       + \sum_{p=1}^{2n_j}\sum_{l=1}^{n}
          \frac{\partial u_j^{p}}{\partial y_l}(y_0,t_0)\frac{\partial u_j^{p}}
          {\partial y_{l+n}}(y_0,t_0)y_ly_{l+n}   \\
       && -
          \sum_{p=1}^{2n_j}\sum_{l=1}^{2n}\frac{\partial u_j^{p}}{\partial y_l}(y_0,t_0)y_lt_p+o(2).
\end{eqnarray*}
Denoting by 
$$\bigl\langle (y,t),(y',t')\bigr\rangle=\sum_{l=1}^{2n}y_ly_l'+\sum_{p=1}^{2n_j}t_p t_p',\quad 
 (y,t),(y',t')\in \mathbb{R}^{2(n+n_j)}$$ the standard real Euclidean inner product, 
we can write the above Taylor expansion in the form
\begin{eqnarray*}
\bigl|\widetilde{\phi}_j(y,t)\bigr|^2 &=& \frac{1}{2}\sum_{p=1}^{2n_j}\biggl|\Bigl\langle \grad_{(y_0,t_0)} \bigl(t_p-u_j^{p}(y,t)\bigr),(y-y_0,t-t_0)\Bigr\rangle \biggr|^2 \\
 && + \frac{1}{2}\sum_{l=1}^{2n}\sum_{p=1}^{n_j} 
 		\biggl( \Bigl(\frac{\partial u_j^{p}}{\partial y_l}(y_0,t_0)\Bigr)^2
    +  \Bigl(\frac{\partial u_j^{p+n_j}}{\partial y_{l}}(y_0,t_0)\Bigr)^2 \biggr) (y_l-y_{0l})^2  \\
    && +             \frac{1}{2}\sum_{p=1}^{2n_j}(t_{p}-t_{0p})^2+o(2).
\end{eqnarray*}    
If $o\bigl(|(y-y_0,t-t_0)|^2\bigr)$ denotes the remainder in the Taylor expansion up to terms of second order of $|\widetilde{\phi}_j|^2$ around a point $(y_0,t_0)$, then for any constant $\epsilon >0$ we can choose a constant $r>0$ such that the estimate 
\begin{equation}\label{ocena12}
\left|\frac{o\bigl(|(y-y_0,t-t_0)|^2\bigr)}{\bigl|(y-y_0,t-t_0)\bigr|^2}\right|<\frac{\epsilon}{2}
\end{equation}
is valid for all pairs $(y_0,t_0)\in Z_j(A\cap V_j')$ and $(y,t)\in B\bigl((y_0,t_0),r\bigr)\Subset U_j'$, where $B\bigl((y_0,t_0),r\bigr)$ is a ball in $\mathbb{C}^{n+n_j}$ centered at $(y_0,t_0)$ and with radius $r$.

Clearly $\grad (t_p-u_j^p(y,t))$ and partial derivatives of $u_j^{p}$ up to second order for $p \in \{1, \ldots, 2n_j\}$ are bounded on $Z_j(A\cap V_j')$. 
Hence the expressions 
\begin{eqnarray}\label{izraz3}
\sum_{p=1}^{2n_j}\biggl|\Bigl\langle \grad_{(y_0,t_0)} \bigl(t_p-u_j^p(y,t)\bigr),\frac{(y-y_0,t-t_0)}{\bigl|(y_0,t_0)\bigr|}\Bigr\rangle\biggr|^2
\end{eqnarray}
and
\begin{eqnarray*}
\frac{\sum_{l=1}^{2n}\sum_{p=1}^{n_j} \Bigl(\bigl(\frac{\partial u_j^{p}}{\partial y_l}(y_0,t_0)\bigr)^2+\bigl(\frac{\partial u_j^{p+n_j}}{\partial y_{l}}(y_0,t_0)\bigr)^2\Bigr)(y_l-y_{0l})^2+\sum_{p=1}^{2n_j}(t_{p}-t_{0p})^2}{|(y-y_0,t-t_0)|^2}
\end{eqnarray*}
are uniformly bounded from above for every $(y_0,t_0)\in Z_j(A\cap V_j')$ and every $(y,t)\in B\bigl((y_0,t_0),r\bigr)$. Therefore, if $r>0$ is chosen small 
enough,
\begin{eqnarray}\label{izraz6}
\frac{\bigl|\phi_j(y,t)\bigr|^2}{\bigl|(y-y_0,t-t_0)\bigr|^2}
\end{eqnarray}
is uniformly bounded from above for all $(y_0,t_0)\in Z_j(A\cap V_j')$ and every $(y,t)\in B((y_0,t_0),r)\subset U_j'$.

Let $\nu_j$ be the normal bundle to $Z_j(A\cap V_j')$ in $T\mathbb{C}^{n+n_j}|_{Z_j(A\cap V_j')}\approx Z_j(A\cap V_j')\times \mathbb{C}^{n+n_j}$ endowed with standard metrics.
Since the gradients of the components of $\phi_j$ generate the normal bundle $\nu_j$ to $Z_j(A \cap V_j')$, the expression (\ref{izraz3})
is positive for every pair of points $(y_0,t_0)\in Z_j(A\cap V_j')$ and $(y,t)\in U_j'$ such that $y=y_0$.
Set 
$$T_j=\{(y,t)\in U_j' \mid y\in z(A\cap \overline{V}_j), \quad \bigl|t-g_j(y)\bigr|<r\}.$$
Clearly the expression (\ref{izraz3}) is then greater than some small positive constant $\epsilon_{0}$ for every $(y_0,t_0)\in Z_j(A\cap \overline{V}_j)$ and for all corresponding $(y,t)\in T_j$ with $y=y_0$. 
Furthermore, the estimate (\ref{ocena12}) implies that the expression (\ref{izraz6})
is greater than $\frac{\epsilon_{0}}{2}$ for all $(y_0,t_0)\in Z_j(A\cap \overline{V}_j)$ and corresponding $(y,t)\in T_j$ $(y=y_0)$, provided that $r$ is choosen small enough. We may also assume that the constant $r$ is independent of $j \in \{1,\ldots,m\}$, hence the above statements hold for every $j$.

Let $(y',t')$ be real local coordinates on $U_k'$ and let $\varphi$ be a biholomorphic transition map between $Z_k(U_k)$ and $Z_j(U_j)$. By choosing $r>0$ sufficiently smaller we can achieve that the expression  
\begin{equation}\label{kvad}
\frac{\bigl|\varphi\bigl(y(x),t(x)\bigr)-\varphi\bigl(y(x_0),t(x_0)\bigr)\bigr|}{\bigl|\bigl(y(x)-y(x_0),t(x)-t(x_0)\bigr)\bigr|}=\frac{\bigl|\bigl(y'(x),t'(x)\bigr)-\bigl(y'(x_0),t'(x_0)\bigr)\bigr|}{\bigl|\bigl(y(x)-y(x_0),t(x)-t(x_0)\bigr)\bigr|}
\end{equation}
is uniformly bounded for every pair of points $x_0\in A\cap V_j'\cap V_k'$ and $x \in X$ such that $\bigl(y(x),t(x)\bigr) \in B\bigl(\bigl(y(x_0),t(x_0)\bigr),r\bigr)$.

According to (\ref{izraz6}) the quotients 
\[
\frac{\bigl|\widetilde{\phi}_k\bigr|^2\bigl(y'(x),t'(x)\bigr)}{\bigl|\bigl(y'(x)-y'(x_0),t'(x)-t'(x_0)\bigr)\bigr|^2} 
\textrm{\qquad and \quad}\frac{\bigl|\widetilde{\phi}_j\bigr|^2\bigl(y(x),t(x)\bigr)}{\bigl|\bigl(y(x)-y(x_0),t(x)-t(x_0)\bigr)\bigr|^2}
\]
are both uniformly bounded from below and from above by some positive constants, provided that $x_0\in A\cap V_j'\cap V_k'$ and $x \in Z_j^{-1}(T_j)\cap Z_k^{-1}(T_k)$ such that $y(x)=y(x_0)$, $y'(x)=y'(x_0)$.
Since the expression (\ref{kvad}) is also bounded with respect to such $x_0$ and $x$, we conclude that the quotient $\frac{|\phi_k|^2}{|\phi_j|^2}$ is bounded in the neighborhood $Z_j^{-1}(T_j)\cap Z_k^{-1}(T_k)$ of $bA \cap \overline{V}_j \cap \overline{V}_k$ in $\Lambda_j \cap \Lambda_k \cap \overline{V}_j \cap \overline{V}_k$ and therefore it is bounded also on $\Lambda_j \cap \Lambda_k \cap \overline{V}_j \cap \overline{V}_k$. It follows that the quotients 
\[
    \frac{\delta+|\phi_k|^2}{\delta+|\phi_j|^2}
\]
are bounded uniformly with respect to $\delta \in (0,1]$ by some constant $M'>0$ on $\Lambda_j \cap \Lambda_k \cap \overline{V}_j\cap \overline{V}_k$.

Since $\tau_j$ tends to $-\infty$ on $bV_j$, we claim that none of the values $\tau_j(x)$ respectively $v_{\delta,j}(x)$ for $x$ sufficiently near $bV_j$ 
contributes to the value $v_{\delta}(x)$ and this property is uniform with respect to $\delta \in (0,1]$. 
First we choose slightly smaller open sets $W_j\Subset V_j$ for every $j\in\{1,\ldots,N\}$ and such that $\Lambda \subset \overline{V} \subset \cup_{j=1} ^{N} W_j$. Clearly there is a constant $M_1>-\infty$ such that $\tau_j >M_1$ on $W_j$ for all $j$. Next we choose larger open sets $W_j'$ with $W_j\Subset W_j'\Subset V_j$ and such that $\tau_j<M_1-M'-2\eta$ on $V_j\backslash W_j'$ for all $j\in\{1,\ldots, N\}$. For any $x\in\Lambda\cap W_k \cap (V_j\backslash W_j')$ and $\delta\in (0,1]$ we then obtain 
\begin{equation}\label{vv}
v_{\delta,j}(x)-v_{\delta,k}(x)=\log\left(\frac{\delta+|\phi_j|^2}{\delta+|\phi_k|^2}\right)+\tau_j(x)-\tau_k(x)<-2\eta.
\end{equation}
By continuity, for every $\delta \in (0,1]$ there exists a neighborhod $V_{\delta} \Subset \cup_{j=1} ^{N} W_j$ of $\Lambda$ and such that $\frac{\delta  +|\phi_j|^2}{\delta+|\phi_k|^2}< M'$ on $V_{\delta}\cap (\overline{V}_j\cap \overline{V}_k)$ for all $k,j \in \{1,\ldots,N\}$.
Hence the inequality (\ref{vv}) above remains valid for every $x\in V_{\delta}\cap W_k \cap (V_j \backslash W_j')$. For any point $x \in \cup_{j=1}^{N}W_{j}$ denote by  $J_{x} \subset \{1,\ldots , N\}$ the nonempty set of all indices $j$ such that $x\in W_j'$. 
It is immediate from Lemma \ref{lemastrogopsh} (\ref{lemmaiii}), that for every $x\in V_\delta$, only the values $v_{\delta,j}(x)$ such that $j \in J_x$ might contribute to the value $v_{\delta}(x)$, and this proves the claim.

Furthermore, by continuity, for every $x\in V_{\delta}\cap W_k \cap (V_j \backslash W_j')$ there is some small neighborhood $V_{x,\delta}$ of $x$ in $V_{\delta}$, and such that the inequality (\ref{vv}) holds for every $y \in V_{x,\delta}$. It follows that only the values $v_{\delta,j}(y)$ with $j \in J_x$ might contribute to the value $v_{\delta}(y)$ for $y \in V_{x,\delta}$. Hence $v_{\delta}$ is of class $\mathcal{C}^2$ on $V_{\delta}$.

Remember that for $x \in V_j \cap \Lambda$ we have $v_{\delta,j}(x)\leq v_{\delta',j}(x)$ for $0<\delta \leq \delta'\leq 1$. Since the set $J_x$ is independent of $\delta \in (0,1]$, Lemma \ref{lemastrogopsh} (\ref{lemai}) then implies that $v_{\delta}$ decreases on $\Lambda$ as $\delta$ approaches to $0$. Observe that the convergence of $v_{0}(x)=\lim_{\delta\to 0}v_{\delta}(x)$ is locally uniform with respect to $\delta \in (0,1]$ on $\Lambda \backslash A$,   and $\{v_0=-\infty\}=A$.

Since the maps $\phi_j$ are holomorphic on $\Lambda_j$, we have the estimate $i\partial\bar{\partial}\log\bigl(\delta+|\phi_j|^2\bigr)\ge 0$ on $\Lambda_j$ and it follows
\[
	i\partial\bar{\partial}v_{\delta,j}(x)= 
	i\partial\bar{\partial} \log\bigl(\delta+|\phi_j(x)|^2\bigr)+ i\partial\bar\partial \tau_j(x) 
\geq i\partial\bar{\partial}\tau_j(x),
\qquad x \in \Lambda_j.
\] 
Clearly $i\partial\bar{\partial}\tau_j$ is bounded from below on the set 
$W_j'$ by some constant $M>-\infty$ for every $j\in\{1,\ldots , N\}$.
The claims proved in previous paragraphs and Lemma \ref{lemastrogopsh} (\ref{lemaiv}) now imply that $M$ is a uniform lower bound for $i\partial\bar{\partial} v_{\delta}$ on the set $\Lambda$. However, we cannot control $i\partial\bar{\partial} v_{\delta}$ from below outside of the set $ \Lambda$ since we do not have a uniform lower bound with respect to $\delta \in (0,1]$ for $i\partial\bar{\partial}\bigl(\log\bigl(\delta+|\phi_j(x)|^2\bigr)\bigr)$ there.
\end{proof}

Observe that in the case of complex curves \cite{lit2} the sets $V_{\delta}$ can be chosen independent of $\delta$. This is true also in the case when $X$ is a complex manifold since we have a global extension of the subvariety $A$ over the boundary to a closed $\mathcal{C}^2$ subvariety (without boundary) and hence we can apply the proof on the extension of the subvariety $A$.

The following tehnical lemma will be needed in the construction of global plu\-ri\-sub\-har\-mo\-nic functions in a neighborhood of $A$ in $X$ (see Lemma \ref{dodatek} below).

\begin{lemma} \label{lematehnicna}
Let $U \Subset \mathbb{C}^{n}$ be an open convex set and let $h \colon U\to\mathbb{R}$ be a $\mathcal{C}^2$ strictly convex function. Set $D=\{z \in U \mid h(z)<0\}\Subset\mathbb{C}^{n}$ and assume $bD\cap U=\{z \in U \mid h(z)=0\}$ and d$h(x)\neq 0$ on $U$. Then for every compact set $L\subset\partial D\cap U$ and every open neighborhood $U_L$ of $L$ in $U$, there exists a $\mathcal{C}^2$ strictly convex function $\widetilde{h}\leq h$ and a convex set $\widetilde{D}=\{z \in U \mid \widetilde{h}(z)<0\}$ such that $D\cup L\subset\widetilde{D}\subset D\cup U_L$ and d$\widetilde{h}\neq 0$ on $b\widetilde{D}\cap U$.
\end{lemma}

\begin{proof}
Let $\chi\in\mathcal{C}_0 ^{\infty}(U_L)$ be a smooth nonnegative function with compact support in $U_L$ and let $\chi >0$ on $L$. For sufficiently small $\epsilon >0$ the function $\widetilde{h}=h-\epsilon\chi$ is a $\mathcal{C}^2$ strictly convex with d$\widetilde{h}(x)\neq 0$ at any point $x\in b\widetilde{D}\cap U$, and it has all other required properties.
\end{proof}

The same argument works also in the case when $h$ is a $\mathcal{C}^2$ strictly plurisubharmonic function. 
See, for example, \cite[Corollary 1.5.20]{lit6}.

\begin{lemma}\label{dodatek}
Let $A$ and $K$ be as in Theorem \ref{trditev} and let $\Lambda$ be the set of the form (\ref{lam}). Then there exist an open neighborhood $V$ of $A\cup K$ in $X$ and a nonnegative plurisubharmonic function $\psi\colon V\to\mathbb{R}_{+}$ such that it vanishes on $A\cup K$ and it is positive on $V\backslash\Lambda$. 
\end{lemma}

\begin{proof}
Recall that for $j\in\{1,\ldots , m\}$ the map $Z_j=(z,w)\colon U_j\to U_j'$ is a holomorphic embedding, $\Gamma_j=\{h_j \leq 0\}$ is a convex set where $h_j$ is a $\mathcal{C}^2$ strictly convex function, and $g_j\colon \Delta^{n} \to\Delta^{n_j}$ is a $\mathcal{C}^2$ map which is holomorphic in the interior of $\Gamma_j$. Moreover, they satisfy the properties (\ref{L1}), (\ref{L3}) and (\ref{L4}) listed in the beginning of this section.

By Lemma \ref{lematehnicna}, applied with 
$L=z(bV_j'\cap bA)$ and $U_L=z(U_j\backslash \overline{V}_j)$, there exists a $\mathcal{C}^2$ strictly convex function $\widetilde{h}_j\leq h_j$ such that $\{\widetilde{h}_j\leq 0\}\supset \Gamma_j$ and $h=\widetilde{h}$ on $z(V_j)$. We denote by $\Gamma_j'$ the component of $\{\widetilde{h}_j\leq 0\}$ which contains $\Gamma_j$;
then $\Gamma_j'\cap z(V_j)=\Gamma_j\cap z(V_j)$.

As in the proof of Proposition \ref{trditev} we can construct a $\mathcal{C}^2$-function $\psi_j\colon U_j\to\mathbb{R}_+$ (replacing the set $\Gamma_j$ by $\Gamma_j'$ in the construction). 
We may use $\widetilde{h}_j$ to obtain a $\mathcal{C}^2$-function $\psi_j'$ with 
$$\psi_j'(z,w)=\widetilde{h}_j(z),\quad (z,w)\in U_j',$$
which is independent of the variable $w$ on $U_j'$ and set $\psi_j=\psi_j' \circ Z_j$. The function $\psi_j$ is then plurisubharmonic in the neighborhood $\widetilde{U}_j$ of the set $\widetilde{\Lambda}_j=\{x\in U_j \mid z(x)\in \Gamma_j'\}$ in $U_j$,
where the set $\widetilde{U}_j$ is of the form (\ref{wU}). We also note that $\{\psi_j\leq 0\}=\widetilde{\Lambda}_j$. Moreover, we may assume $\{\psi_j= 0\}=\widetilde{\Lambda}_j$, since we can allways compose $\psi_j$ by a smooth convex increasing function $\chi$ such that $\chi(t)=0$ for $t\leq 0$ and $\chi(t)>0$ for $t>0$. 
By choosing the set $\widetilde{U}_j$ small enough, we can insure that the closures of the sets $\widetilde{U}_j\backslash (\widetilde{\Lambda}_j \cup V_j')$ and $V_j'$ are disjoint.
Therefore, we can extend the restriction $\psi_j |_{V_j'\cap \widetilde{U}_j}$ by $0$ along 
$\widetilde{U}_{j} \backslash V_j'$ and outside of $U_j$.
We obtain a nonnegative plurisubharmonic function $\widetilde{\psi}_j$ that 
vanishes on $\Lambda_{j}\subset \widetilde{\Lambda}_j$ (of the form (\ref{lamj})) and outside of the set $U_j$. Since $U_j \cap A \subset\Lambda_j$ and $K$ does not intersect 
$U_j \backslash \Lambda_j$, we see that the zero set of $\widetilde{\psi}_j$ contains $A \cup K$. Further as we have $\widetilde{\Lambda}_j \cap V_j = \Lambda_j \cap V_j$ and as $\widetilde{\psi}_j$ agrees with $\psi_j$ on $V_j$, it follows that $\widetilde{\psi}_j$ is positive on $(\widetilde{U}_j \cap V_j) \backslash \Lambda_j$.

Choose a neighborhood $V$ of $A\cup K$ such that $V \cap U_j \subset \widetilde{U}_j$ and set 
$$\psi=\sum_{j=1} ^{m}\widetilde{\psi}_j \colon V \to \mathbb{R}_+.$$ 
This function is plurisubharmonic and it vanishes on $A \cup K$; it remains to show that it is positive on $V\backslash\Lambda$. To see this, let $\Lambda_j'$ be the set of the form (\ref{lamj'}) and we observe that  
\[
(V \cap V_j)\backslash \Lambda_j'=\bigl((V \cap V_j)\backslash \Lambda_j \bigr)\cup 
 \bigl( (V\cap V_j)\cap \bigcup_{k\neq j}(V\cap V_k) \backslash \Lambda_k \bigr).
\]
Since every $\widetilde{\psi}_j$ $(j \in \{1,\ldots,m\})$ is positive on $(V \cap V_j) \backslash \Lambda_j$, we see that $\psi$ is positive on every set $(V \cap V_j)\backslash \Lambda_j'$ and hence also on $V\backslash \Lambda$.
\end{proof}

\begin{proof}[Proof of Theorem \ref{izrekg}]
Choose the set $V$ as in Lemma \ref{dodatek}. Given an open 
neighborhood $U$ of the set $A\cup K$ with $\overline U\subset V$, 
we must find an open Stein neighborhood of $A\cup K$ contained in $U$.

Using Lemma \ref{lemaspmeja} we get an open set $W\subset X$ 
and a strictly plurisubharmonic $\mathcal{C}^2$-function 
$\rho\colon \overline{W}\to\mathbb{R}$ such that 
$A\cup K\subset W\Subset U$, $\rho|_{bW}>0$, and $\rho |_K<0$. 
We can assume that $i\partial\bar{\partial}\rho\geq c>0$ on $\overline{W}$ for some constant $c>0$.

Let $v_{\delta}\colon V_{\delta}\to\mathbb{R}$ for $\delta\in [0,1]$ be the family of functions furnished by Lemma \ref{trditev2}. Thus $i\partial\bar{\partial}v_\delta >M$ on the set $\Lambda\subset V_{\delta}$ of the form (\ref{lam}) for some constant $M>-\infty$.
As $\delta$ decreases to $0$, the functions $v_{\delta}$ for $\delta \in (0,1]$ decrease monotonically to the function $v_0$ and $\{v_0=-\infty\}=A$. By subtracting a constant from all $v_{\delta}$ we can achieve that $v_{\delta}\leq v_1<0$ on $K$ for every $\delta\in [0,1]$. Set 
\[
\rho_{\varepsilon,\delta}=\rho + \varepsilon v_{\delta}\colon \overline{W} \cap V_{\delta}\to \mathbb{R}.
\]
We take a sufficiently small $\varepsilon >0$ such that $\rho_{\varepsilon,\delta}\geq \rho_{\epsilon,0}>0$ on $bW\cap \Lambda$ for all $\delta \in (0,1]$, and such that $c+\varepsilon M>0$. Since $i\partial\bar{\partial}v_\delta >M$ on $\Lambda$, it follows that $\rho_{\varepsilon,\delta}$ is strictly plurisubharmonic function on $\Lambda\cap\overline{W}$ for all $\delta \in (0,1]$. 
Let $M'<M$ be a slightly smaller constant such that $c+\varepsilon M'>0$ still holds. By continuity we have $i\partial\bar{\partial}v_\delta >M'$ on some neighborhood of $\Lambda$. Hence, after shrinking $V_{\delta}$ around $\Lambda$, we may assume that $\rho_{\epsilon,\delta}$ is strictly plurisubharmonic also on a neighborhood $V_{\delta}\cap \overline{W}$ of the set $\Lambda\cap\overline{W}$ in $\overline{W}$.

We now fix $\varepsilon$ and choose $\delta>0$ small enough in order to make $\rho_{\varepsilon,\delta}$ negative on $A$. Therefore $A \cup K$ is contained inside the zero sublevel set of $\rho_{\varepsilon,\delta}$ 
(see Figure \ref{slika4}). 
Note that $\rho_{\epsilon,\delta}<0$ on $K$, since $\rho$ and $v_{\delta}$ are both negative on $K$. We 
keep $\varepsilon$ and $\delta$ fixed for the rest of the proof.

\begin{figure}[ht]
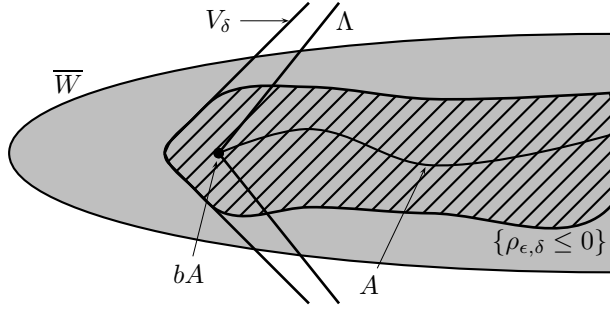

\psset{unit=0.8 cm}

\pspicture*(1,-1)(11,5)
\psellipse[fillstyle=solid,fillcolor=lightgray](11,2)(10,2)
\psline[linewidth=1pt,linearc=.15]{*-}(4.5,2)(6.5,4.5)
\psline[linewidth=1pt,linearc=.15]{*-}(4.5,2)(6.5,-0.5)
\psline[linewidth=1pt,linearc=.25]{-}(6,4.5)(3.5,2)(6,-0.5)
\pscurve[showpoints=false]{-}(4.5,2)(6,2.4)(8,1.8)(11,2.3)
\pscustom[linewidth=1pt,fillstyle=hlines]{
\pscurve[showpoints=false]{-}(3.6,2)(3.7,2.2)(4,2.5)(4.6,3)(6,3.1)(8,2.9)(11,3)
\pscurve[showpoints=false]{-}(11,1.1)(8,1)(6,1.1)(4.6,1)(4,1.5)(3.7,1.8)(3.6,2)}
\rput(2,3.2){$\overline{W}$}
\rput(7,-0.2){\rnode{g}{$A$}}
\rput(8,1.9){\rnode{h}{}}
\ncline[nodesep=3pt,  linewidth=0.3pt]{->}{g}{h}
\rput(4,0){\rnode{e}{$bA$}}
\rput(4.5,2){\rnode{f}{}}
\ncline[nodesep=3pt, linewidth=0.3pt]{->}{e}{f}
\rput(4.5,4.2){\rnode{a}{$V_{\delta}$}}
\rput(5.8,4.2){\rnode{b}{}}
\ncline[nodesep=3pt, linewidth=0.3pt]{->}{a}{b}
\rput(6.6,4.2){\rnode{c}{$\Lambda$}}
\rput(10,0.5){$\{\rho_{\epsilon,\delta}\leq 0\}$}
\endpspicture
\caption{The zero sublevel set of $\rho_{\varepsilon,\delta}$}
\label{slika4}
\end{figure}

By adding a fast growing plurisubharmonic function $\widetilde{\psi}$ described in the next 
paragraph, we shall  round off the set $\{\rho_{\varepsilon,\delta}\leq 0\}$ sufficiently close to $bA$. Indeed, we want the zero sublevel set $\{\rho_{\varepsilon,\delta}+\widetilde{\psi} <0\}$ 
to be a relatively compact subset included inside the region of plurisubharmonity of the function $\rho_{\varepsilon,\delta}+\widetilde{\psi}$. In order to keep $A \cup K$ inside the zero sublevel set, $\widetilde{\psi}$ will have to be zero on $A\cup K$.

Let $\psi$ be the function furnished by Lemma \ref{dodatek}; thus $\psi$ is nonnegative plurisubharmonic on $V$, 
it vanishes on $A\cup K$, and it is positive on $V\backslash\Lambda \supset \overline{W}\backslash \Lambda$. 
Choose an open set $W'\subset W$ such that $W'\Subset V_{\delta}$ and $bW\cap\Lambda=bW'\cap\Lambda$.
As $\rho_{\epsilon,\delta}>0$ on $bW'\cap\Lambda=bW\cap\Lambda$, there is a neighborhood $\widetilde{V}\subset \overline{W'}$ of $bW'\cap \Lambda$ such that $\rho_{\epsilon,\delta}>0$ on $\widetilde{V}$. 
As the closure of a relatively compact set $bW'\backslash \widetilde{V}$ is disjoint from $\Lambda$ and $\psi$ is positive on $V\backslash\Lambda \supset \overline{W'}\backslash\Lambda$, there is a constant $m_1>0$ such that $\psi>m_1$ on $bW'\backslash \widetilde{V}$. Let $M_1$ be another constant such that 
$\rho_{\varepsilon,\delta}<M_1$ on the set $bW'\backslash \widetilde{V}$.
Choose a sufficiently large constant $C>0$ such that $Cm_1+M_1>0$
and set $$ \widetilde{\psi}(x)=C\cdot \psi(x), \quad x \in V.$$
The choice of $C$ implies that 
$\rho_{\varepsilon,\delta}(x) +\widetilde{\psi}(x)>0$ for $x\in bW'$. 
Therefore the set 
\[
	\omega=\{x\in W'\mid \rho_{\varepsilon,\delta}(x) +\widetilde{\psi}(x)< 0\}\Subset W'
\]
lies inside the region where $\rho_{\varepsilon,\delta}+\widetilde{\psi}$ is strictly plurisubharmonic
(see Figure \ref{slika5}). Since $\widetilde{\psi}$ vanishes on $A\cup K$ and $\rho_{\varepsilon,\delta}<0$
on $A\cup K$, we have $A\cup K\subset \omega$.

\begin{figure}[ht]
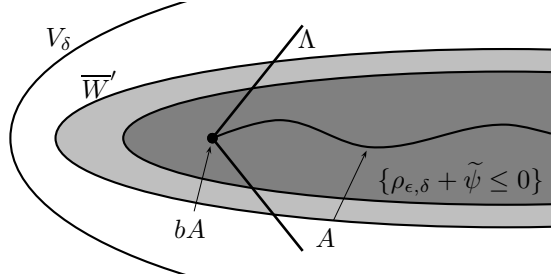

\psset{unit=0.6 cm}

\pspicture*(0,-1)(12,5)
\psellipse[fillstyle=solid,fillcolor=lightgray](11,2)(10,2)
\psellipse[fillstyle=solid,fillcolor=gray](11,2)(8.5,1.5)
\psline[linewidth=1pt,linearc=.15]{*-}(4.5,2)(6.5,4.5)
\psline[linewidth=1pt,linearc=.15]{*-}(4.5,2)(6.5,-0.5)
\psellipse(11,2)(11,4)
\pscurve[showpoints=false]{-}(4.5,2)(6,2.4)(8,1.8)(11,2.3)(12,2.1)
\rput(2,3.2){$\overline{W}'$}
\rput(7,-0.2){\rnode{g}{$A$}}
\rput(8,1.9){\rnode{h}{}}
\ncline[nodesep=3pt,  linewidth=0.3pt]{->}{g}{h}
\rput(4,0){\rnode{e}{$bA$}}
\rput(4.5,2){\rnode{f}{}}
\ncline[nodesep=3pt, linewidth=0.3pt]{->}{e}{f}
\rput(1.1,4.2){$V_{\delta}$}
\ncline[nodesep=3pt, linewidth=0.3pt]{->}{a}{b}
\rput(6.6,4.2){\rnode{c}{$\Lambda$}}
\rput(10,1.1){$\{\rho_{\epsilon,\delta}+\widetilde{\psi}\leq 0\}$}
\endpspicture
\caption{The zero sublevel set of $\rho_{\varepsilon,\delta}+\widetilde{\psi}$}
\label{slika5}
\end{figure}

Finally, as $\rho_{\varepsilon,\delta} +\widetilde{\psi}$ is strictly plurisubharmonic bounded exhaustion function on $\omega$, the set $\omega$ is Stein by Narasimhan's theorem \cite{Narasimhan} (see also \cite{lit9,Narasimhan2}).

This completes the proof of Theorem \ref{izrekg}.
\end{proof}

\section{Tubular Stein neighborhoods}\label{TS}

Let $D$ be a relatively compact domain with $\mathcal{C}^2$-boundary 
in a complex manifold $S$. Given an integer $r\geq 0$ and a complex manifold $X$, 
we denote by $\mathcal{C}^r(\overline{D},X)$ the set of all 
$\mathcal{C}^r$-maps $\overline{D} \to X$ and by 
\[
	\mathcal{A}^r(D,X) = \{f\in \mathcal{C}^r(\overline{D},X) \colon
	f|_D \in \mathcal{O}(D,X)\}
\]
the set of all $\mathcal{C}^r$-maps $\overline{D} \to X$ that are holomorphic on $D$.

A complex vector bundle $E\to\overline{D}$ is said to be 
of class $\mathcal{A}^r(D)$, if it is 
of class $\mathcal{C}^r$ over $\overline{D}$ and holomorphic 
over the interior $D=\overline{D}\backslash bD$.
We shall identify the base $\overline D$ with the zero section
of the total space $E$.

The following is the first main result of this section.

\begin{trditev}\label{trditev22}
Let $D\Subset S$ be a relatively compact domain 
with $\mathcal{C}^2$-boundary  in a Stein manifold $S$,
and let $X$ be a complex manifold.
Assume that $f\colon\overline{D}\hookrightarrow X$ is a 
$\mathcal{C}^r$-embedding $(r\geq 2)$ that is holomorphic in $D$, 
and let $\nu$ denote the (complex) normal bundle 
of $A=f(\overline D)$ in $X$. Then there exist
\begin{enumerate}
\item an open neighborhood $U_{\overline{D}}$ of $\overline{D}$ in $S$, 
\item a holomorphic vector bundle 
$\pi \colon \theta\to U_{\overline{D}}$, 
\item 
a $\mathcal{C}^{r-1}$-isomorphism of complex vector bundles
$\tau \colon f^* \nu \to \theta|_{\overline{D}}$ over $\overline D$ that is holomorphic over $D$, (By $\theta|_{\overline{D}}$ we denoted the restricion of $\theta$ to $\overline{D}$.) 
\item an open neighborhood $\Omega\subset\theta$, with convex fibers, 
of the zero section $\overline D$ (of the restricted bundle 
$\theta|_{\overline{D}}$), and 
\item a $\mathcal{C}^{r}$-diffeomorphism $F\colon \Omega \to F(\Omega) \subset X$ 
that is holomorphic on $\pi^{-1}(D)\cap \Omega$ and such that $F(\overline D)=A$.
(Here $\overline D$ is again the zero section of $\theta|_{\overline{D}}$.)
\end{enumerate}
\end{trditev}

\begin{proof}
By Theorem \ref{izrekg} the set $A=f(\overline{D})$ admits an open Stein neighborhood $\omega \subset X$. 
By Cartan's Theorem $A$ then there exist holomorphic vector fields $v_1, \ldots, v_k$ on $\omega$ that generate the 
tangent space $T_xX \approx T^{(1,0)}_x X$ at every point $x\in \omega$.
The flow $\varphi_t ^j$ of the vector field $v_j$ is well defined and holomorphic on some open relatively compact neighborhood $\omega_0$ of $A$ in $\omega$ and for sufficiently small values of $t\in\mathbb{C}$. Indeed, for any point $x\in\omega$ there is a constant $T_{x}$ such that the solution $\varphi_{t} ^j(x)$ of the ordinary differential equation $\frac{d}{dt}\varphi_{t} ^j(x)=v_j(\varphi_t ^j(x))$ with the initial condition $\varphi_0^j(x)=x$ exists for all $t\in\mathbb{C}$ with $|t|<T_{x}$. For a relatively compact neighborhood $\omega_0\Subset\omega$ of $A$ we can choose a constant $T$ such that $0<T<T_{x}$ for every $x\in\omega_0$.

We extend $f\colon \overline{D}\to X$ to a $\mathcal{C}^r$-map from a neighborhood of $\overline{D}$ in $S$ 
such that the extended map $f$ is $\overline{\partial}$-flat to order $r-1$ on $\overline{D}$ 
(i.e., $\overline{\partial}f$ and its derivatives $D^{\alpha}(\overline{\partial}f)$ for 
$|\alpha|\leq r-1$ vanish on $\overline{D}$). The map 
\[
   F(z,t_1,t_2,\ldots, t_k)=\varphi_{t_1} ^{1}\circ \cdots \varphi_{t_k} ^{k}\circ f(z)
\]
is then well defined in some neighborhood of $\overline{D}\times \{0\}^k$ in $S\times \mathbb{C}^k$. Furthermore, 
$F$ is a $\mathcal{C}^r$-map that is holomorphic in the variables $(t_1, \ldots,t_k)$ and satisfies 
$\frac{\partial F}{\partial \overline{z}}(z,t)=0$ for $z\in \overline D$. 
Clearly we have $F(z,0)=f(z)$ and 
$\frac{\partial F}{\partial t_j}(z,0))=v_j(F(z,0))=v_j(f(z))$.
As the vectors $v_1(x), \ldots, v_k(x)$ generate $T_x X$ for all $x \in \omega$, 
$F$ is a submersion at every point of $\overline{D}\times\{0\}^k$.

For every $z\in\overline D$ we denote by $\Xi_z$ the partial diferential 
$D_t F(z,t)|_{t=0}  \colon T_0\mathbb{C}^k\approx \mathbb{C}^k\to T_{f(z)}X$ 
of the map $t\mapsto F(z,t)$ at $t=0$. 
Let $E\subset\overline{D}\times \mathbb{C}^k$ be the complex vector subbundle 
with fibers
\begin{equation}\label{vlakno}
E_z=\{v\in\mathbb{C}^k\mid \Xi_z(v)\in T_{f(z)}A\}, \quad z\in \overline{D}.
\end{equation}
Note that $E$ is of class $\mathcal{A}^{r-1}(D)$ since the tangent bundle 
$TA$ of $A$ is of this class.

We claim that $E$ is complemented, in the sense that there exists a complex vector subbundle 
$\widetilde \theta \subset \overline{D}\times \mathbb{C}^k$ of class $\mathcal{A}^{r-1}(\overline{D})$ 
such that $E \oplus \widetilde \theta = \overline{D}\times \mathbb{C}^k$.  
To see this, consider the short exact sequence of $\mathcal{A}^{r-1}(D)$-vector bundles 
\begin{equation}\label{ezo}
	0 \to E \hookrightarrow \overline{D}\times \mathbb{C}^k \stackrel{q}{\rightarrow} \widetilde E \to 0,
\end{equation}
where $E\hookrightarrow \overline{D}\times \mathbb{C}^k$ is an inclusion, 
$\widetilde E=(\overline{D}\times \mathbb{C}^k)\slash E$ 
is the quotient bundle, and $q$ is the  quotient projection. 
Next, consider the associated short exact sequence of locally free sheaves of 
$\mathcal{A}^{r-1}(D)$-sections 
\begin{equation}\label{ezo2} 
	0 \to \mathcal{S}(E) \to \mathcal{S}(\overline{D}\times \mathbb{C}^k) \to \mathcal{S}(\widetilde E) \to 0,
\end{equation}
where we denoted by $\mathcal{S}(E)$, $\mathcal{S}(\overline{D}\times \mathbb{C}^k)$ and $\mathcal{S}(\widetilde{E})$ respectively the sheaves of $\mathcal{A}^{r-1}(D)$-sections of bundles $E$, $\overline{D}\times \mathbb{C}^k$ and $\widetilde{E}$. For the details of the equivalence between the category of vector bundles and the category of locally free sheaves (of sections) we refer to \cite{lit4} and \cite{lit21}.

Applying Cartan's Theorem B for locally free sheaves of class $\mathcal{A}^0(D)$
over a strongly pseudoconvex domain (see \cite{lit20,lit19}), 
we can follow the usual proof for locally free sheaves on a Stein manifold 
(see e.g.\ \cite[VIII/7.Theorem]{lit21}) 
to obtain a direct sum splitting 
$\mathcal{S}(\widetilde{E}) \oplus \mathcal{S}(E)=\mathcal{S}(\overline{D}\times \mathbb{C}^k)$ 
of the above exact sequence. 
Indeed, the exactness of the sequence (\ref{ezo2}) implies the exactness of
\[
	0\to\Hom\bigl(\mathcal{S}(\widetilde{E}\bigr),\mathcal{S}(E))
	\stackrel{\alpha}{\rightarrow}\Hom\bigl(\mathcal{S}(\widetilde{E}),\mathcal{S}(\overline{D}
	\times \mathbb{C}^k)\bigr)\stackrel{\beta}{\rightarrow}
	\Hom\bigl(\mathcal{S}(\widetilde{E}),\mathcal{S}(\widetilde{E})\bigr)\to 0.
\]
As $H^1(\overline{D},\Hom(\mathcal{S}(\widetilde{E}),\mathcal{S}(E)))=0$ by the version of Cartan's theorem B discussed in the beginning of the paragraph, we see that the map 
\[
    H^0\bigl(\overline{D},\Hom\bigl(\mathcal{S}(\widetilde{E}),\mathcal{S}(\overline{D}\times \mathbb{C}^k)\bigr)\bigr)\stackrel{\overline{\beta}}{\rightarrow} H^0\bigl(\overline{D},\Hom\bigl(\mathcal{S}(\widetilde{E}),\mathcal{S}(\widetilde{E})\bigr)\bigr)
\]
is surjective. Hence there exists 
$h\in H^0\bigl(\overline{D},\Hom\bigl(\mathcal{S}(\widetilde{E}),\mathcal{S}(\overline{D}\times \mathbb{C}^k)\bigl)\bigl)$ such that $\overline{\beta}(h)=\beta \circ h$ equals the identity. Then 
$h\bigl(\mathcal{S}(\widetilde E)\bigr)$ gives us the corresponding vector subbundle $\widetilde{\theta}$ of $\overline D\times\C^k$, and it is 
immediate that $\overline D\times\C^k\approx E\oplus \widetilde \theta$
as an $\mathcal{A}^{r-1}(D)$-isomorphism.

By the result of Heunemann (see \cite[Theorem 1]{lit22}) we can approximate the subbundle 
$\widetilde{\theta} \subset \overline D \times\C^k$
uniformly on $\overline{D}$ by a complex vector subbundle $\theta\subset U_{\overline{D}}\times \mathbb{C}^k$ that is holomorphic over an open set $U_{\overline{D}}\supset \overline{D}$ in $S$. 
More precisely, there exists a continuous complex vector bundle automorphism $\tau'$ of the trivial bundle $\overline{D}\times \mathbb{C}^k$ that is close to the identity over $\overline{D}$, holomorphic over $D$, 
and such that $\theta|_{\overline{D}}=\tau'(\widetilde{\theta})$. In addition, the morphism $\tau'$ can be chosen as close to the identity morphism as desired, and hence we may assume that $\overline{D}\times \mathbb{C}^k=E \oplus \theta|_{\overline{D}}$. For a simpler proof of this result see 
the Appendix in \cite{lit2} and note that it 
additionally gives us an automorphism of class $\mathcal{A}^{r-1}(D)$.

Since $E_z$ must contain the kernel of the map $\Xi_z$ for any $z\in \overline{D}$ (see (\ref{vlakno})), the restriction $\Xi |_{\widetilde{\theta}} \colon \widetilde{\theta} \to \Xi(\widetilde{\theta}) \subset TX|_A$ is an injective vector bundle map and 
\[
   \Xi(\widetilde{\theta}) \oplus TA=TX|_A.
\]
If $\theta$ is a sufficiently good approximation of the bundle $\widetilde{\theta}$, it follows that the restriction $\Xi_z|_{\theta_z} \colon \theta_z \to \Xi(\theta_z)$ is also an isomorphism for every $z \in \overline{D}$ and $\Xi(\theta) \oplus TA=TX|_A.$ Thus the normal bundle $\nu$ of $A$ in $X$ is $\mathcal{A}^{r-1}(D)$-isomorphic to the restricted bundles $\widetilde{\theta}|_{\overline{D}}$ and $\theta_{\overline{D}}$. Moreover, there exists a complex vector bundle isomorphism 
$\tau \colon f^* \nu \to \theta|_{\overline{D}}$ of class $\mathcal{A}^{r-1}(D)$.

Remember that $F(z,0)=f(z)$ is a $\mathcal{C}^r$-embedding.
By the inverse mapping theorem there is an open neighborhood $\Omega$ in $\theta$ of the zero section
$\overline{D}$ in $\theta|_{\overline{D}}$ and 
such that $F$ maps $\Omega$ $\mathcal{C}^r$-diffeomorphically onto an open 
neighborhood $F(\Omega)$ of $A$ in $X$, and $F$ is biholomorphic on 
$\Omega \cap \pi^{-1}(D)$. By shrinking $\Omega$ we may insure further that 
$\Omega$ has convex fibers and that $F(\Omega)\subset \omega$.

On the fiber $\{z\} \times \mathbb{C}^k$ of the bundle $\overline{D} \times \mathbb{C}^k$ we have a unique decomposition $t=t'\oplus t''\in E \oplus \theta$. Note that the partial diferential $D_{t''}F_{(z,0)}$ is also an isomorphism for $z \in \overline{D}$ and hence $F$ is fiberwise biholomorphic on $\Omega \cap \pi^{-1}(D)$.

This completes the proof of Proposition \ref{trditev22}.
\end{proof}

\begin{remark}
In connection to Proposition \ref{trditev22} (check also Theorem \ref{izrek2}) we mention Theorem $1.1.$ in \cite{lit25} which states the following: If $h\colon X \to S$ is a holomorphic submersion and $f\colon \overline{D}\to X$ is a continuous section that is holomorphic in $D$, then 
there exists a holomorphic vector bundle $\xi \to U_{\overline{D}}$ over an open neighborhood $U_{\overline{D}}$ of $\overline{D}$ in $S$,  and for every open set $\widetilde{\Omega} \supset f(\overline{D})$ in $X$ there exists an open Stein set $\Omega \subset \widetilde{\Omega}$ containing $f(\overline{D})$ and a fiber-preserving biholomorphic map from $\Omega$ onto an open subset with convex fibers in $\xi$. This result is proved by the method of holomorphic sprays developed in \cite{lit11} and \cite{lit2}.
\qed \end{remark}

Let us now recall the definitions and establish the notation considering norms on $L^{\infty}$-spaces on manifolds.
Let $W\Subset X$ be an open relatively compact subset in a complex manifold $X$. For $N \in \mathbb{N}$ and a map 
$f\in \mathcal{C}^0(\overline{W},\mathbb{C}^N)$ we set 
\[
\|f\|_{L^{\infty}(W)}=\sup_{z\in W} |f(z)|,
\]
where $|f(z)|$ denotes the Euclidean norm on the space $\mathbb{C}^N$. Let now $\alpha$ be a $(0,1)$-form on $\overline{W}$ with coefficients of class $\mathcal{C}^1(\overline{W},\mathbb{C}^N)$. In the case $X=\mathbb{C}^n$ we define  
\[
\|\alpha\|_{L_{0,1}^{\infty}(W)}=\sup_{z\in W} \sum_j |\alpha_j(z)|,
\]
where $\alpha(z)=\sum_j \alpha_{j}(z)d\overline{z}_j$ with $\alpha_i\in \mathcal{C}^1(\overline{W},\mathbb{C}^N)$. In general we choose an open cover $V_1, \ldots V_m$ of $\overline{W}$ and such that each $V_j$ is a relatively compact subset in some holomorphic coordinate patch in $X$. As one can compute $\|\alpha\|_{L^{\infty}(W\cap V_j)}$ for $j\in \{1,\ldots,m\}$ in the coordinate patch in $\mathbb{C}^n$, we  set a norm
\[
\|\alpha\|_{L_{0,1}^{\infty}(W)}=\sum_{j=1}^m \|\alpha\|_{L_{0,1}^{\infty}(W\cap V_j)}.
\]
Note that any other cover yields an equivalent norm.

Another way to define this norm is to choose a locally finite open cover $\{U_j\}_j$ of $X$ and a partition of unity  $\{\vartheta_j\}_j$ subordinate to this cover. We set $\|\alpha\|_{L_{0,1}^{\infty}(W)}=\sup_{z \in W}\sum_j \vartheta_j(z)|\alpha_j(z)|$, where $\alpha(z)=\sum_j \alpha_{j}(z)d\overline{z}_j$ is a $(0,1)$-form with respect to coordinates $z$ in $U_j$. Since $\overline{W}$ is compact, different choices of covers and partitions of unity give equivalent norms. In a similar fashion we define the norm for any differential form.

The following local result is a modification of Lemma 3.2. in \cite{lit24} and for the sake of completeness we shall rewrite its proof.

\begin{lemma}\label{tl}
Let $\epsilon\mathbb{B}$ be the ball in $\mathbb{C}^n$ centered at $0$ and with radius $\epsilon$, and let $v\in\mathcal{C}^{r+1}(\epsilon \mathbb{B})$ $(r\in\mathbb{N})$. Then there is a constant $c<\infty$ such that
\[
|D^{\alpha}v(0)|\leq c\, \bigl(\epsilon||D^{\alpha}\bar{\partial}v||_{L^{\infty}(\epsilon \mathbb{B})}+\epsilon^{-|\alpha|}||v||_{L^{\infty}(\epsilon \mathbb{B})}\bigr), \qquad |\alpha|\leq r.
\]
\end{lemma}

For a multiindex $\alpha=(\alpha_1,\ldots, \alpha_{2n})$ we denote by $D^{\alpha}=\frac{\partial^{|\alpha|}}{\partial x_1^{\alpha_1}\ldots \partial y_{2n}^{\alpha_{2n}}}$ the partial derivative with respect to coordinates $(x_1,y_1,\ldots,x_n,y_n)$ on $\mathbb{R}^{2n}\approx \mathbb{C}^n$.

\begin{proof}
We apply the Bochner-Martinelli formula 
\[
	g(z)=\int_{\partial T}g(\zeta)B(\zeta,z)-\int_{T}\bar{\partial}g(\zeta)\land B(\zeta,z),
\]
which is valid for any open subset $T\Subset \mathbb{C}^n$ with piecewise $\mathcal{C}^1$-boundary and any  $g\in\mathcal{C}^1(\overline{T})$, where $B(\zeta,z)$ is the Bochner-Martinelli kernel 
\[
	B(\zeta,z) = c_n\sum_{j=1} ^{n}(-1)^{j-1}
	\frac{\overline{\zeta_j-z_j}}{|\zeta-z|^{2n}} d\bar{\zeta}[j] \land d\zeta.
\]

Let $\chi\colon \mathbb{R}\to [0,1]$ be a cut-off function with $\chi(t)=1$ when $|t|\leq \frac{1}{2}$ and $\chi(t)=0$ when $|t|\geq 1$. For $w\in\mathbb{C}^n$ we set $\chi_{\epsilon}(w)=\chi(\frac{2|w|}{\epsilon})$. 
Its partial derivatives satisfy 
$|D^{\alpha}\chi_{\epsilon}|\leq C_{\alpha}\epsilon^{-|\alpha|}$ for all multiindices $\alpha$, where $C_{\alpha}>0$ is a constant depending on $\alpha$.

Applying the Bochner-Martinelli formula to $g(\zeta)=\chi_{\epsilon}(\zeta-z)v(\zeta)$ on $\epsilon\mathbb{B}$ we obtain
\[
   v(z)=-\int_{\epsilon\mathbb{B}}{\bar{\partial}_{\zeta}\bigl(\chi_{\epsilon}(\zeta-z)v(\zeta)\bigr)\land B(\zeta,z)}=
\]   
\[
   =-\int_{\epsilon\mathbb{B}}{\bar{\partial}v(\zeta)\land \chi_{\epsilon}(\zeta-z)B(\zeta,z)}-\int_{\epsilon\mathbb{B}}{v(\zeta)\bar{\partial}_{\zeta}\chi_{\epsilon}(\zeta-z)\land B(\zeta,z)}.
\]  
These are convolution operators and we may differentiate on either integrand. This gives for $|\alpha|\leq r$:
\[
   D^{\alpha}v(z)=-\int_{\epsilon\mathbb{B}}{D^{\alpha}\bar{\partial}v(\zeta)\land \chi_{\epsilon}(\zeta-z)B(\zeta,z)}-\int_{\epsilon\mathbb{B}}{v(\zeta)\partial_z ^{\alpha}\bigl(\bar{\partial}\chi_{\epsilon}(\zeta-z)\land B(\zeta,z)\bigr)}=
\]   
\[
   =I_1(z)+I_2(z).
\]   
Using $|B(\zeta,z)|\leq c|\zeta-z|^{1-2n}$ for some positive constant $c$, we can estimate the integrals as follows:
\[
   |I_1(z)|\leq \int_{|\zeta-z|\leq \frac{\epsilon}{2}}|D^{\alpha}\bar{\partial}v(\zeta)||\zeta-z|^{1-2n}dV\leq c\epsilon \|D^{\alpha}\bar{\partial}v\|_{L^{\infty}(\epsilon\mathbb{B})}
\]   
and
\[
   |I_2(z)|\leq \int_{\frac{\epsilon}{4}\leq |\zeta-z|\leq \frac{\epsilon}{2}}|v(\zeta)|\epsilon^{-2n-|\alpha|}dV\leq c\epsilon^{-|\alpha|} \|v\|_{L^{\infty}(\epsilon\mathbb{B})}.
\]   
\end{proof}

To prove Theorem \ref{izrek2} we shall correct the diffeomorphism $F$,
furnished by Proposition \ref{trditev22}, 
to get a biholomorphism by solving the $\overline\partial$-equation with estimates. 
We actually prove the following more precise result that includes Theorem \ref{izrek2}.

\begin{theorem}
\label{izrek3}
Let $D\Subset S$ be a relatively compact domain with $\mathcal{C}^r$-boundary 
$(r\ge 2)$ in a Stein manifold $S$, and let 
$f\colon\overline{D}\hookrightarrow X$ be an embedding of
class $\mathcal{A}^r(D,X)$ to a complex manifold $X$.
Let $\theta$ be a holomorphic vector bundle furnished by Proposition \ref{trditev22},
and identify $\overline D$ with the zero section of $\theta|_{\overline D}$.
Then for every open neighborhood $U$ of $A=f(\overline D)$ in $X$ 
there exist an open Stein domain $\omega\subset \theta$ containing $\overline D$
and a biholomorphic map $\Phi\colon \omega\to\omega' \subset X$ 
such that $A\subset \omega' \subset U$, and such that 
$\Phi^{-1}(A)$ is the graph of an $\mathcal{A}^r$-section of the restricted bundle $\theta|_{\overline{D'}}$ 
over the closure $\overline{D'}$ of an open strongly pseudoconvex domain $D'\subset S$.
\end{theorem}

\begin{proof}
Choose an open set $U\subset X$ containing the image $A=f(\overline D)$.
By Theorem \ref{izrekg} there is an open Stein domain $\omega$ 
in $X$ with $A\subset \omega \subset U$.

Proposition \ref{trditev22} furnishes a holomorphic vector bundle 
$\pi \colon \theta\to U_{\overline{D}}$ 
over a neighborhood $U_{\overline{D}}$ of $\overline{D}$ in $S$,
an open neighborhood $\Omega\subset \theta$ of the zero section 
$\overline D$ of the restricted bundle $\theta_{\overline{D}}$,
and a $\mathcal{C}^{r}$-diffeomorphism $F \colon \Omega \to F(\Omega) \subset \omega$ 
such that $F$ is holomorphic on the set 
\[
	\Omega_0:=\pi^{-1}(D) \cap \Omega \subset \theta.
\]
Recall that $\theta$ is a holomorphic subbundle of a trivial vector bundle 
$U_{\overline{D}} \times \mathbb{C}^k$ for some 
$k \in \mathbb{N}$, and hence we can identify each element of $\theta$ with
a unique point $(z,t)\in U_{\overline{D}} \times \mathbb{C}^k$.

Our goal is to approximate the diffeomorphism $F$ sufficiently well 
in a neighborhood of $\overline D$ by biholomorphic maps.
Since $F(\overline D)=A$ admits a Stein neighborhood $\omega$ in $X$,
we can reduce this nonlinear approximation problem to the linear problem
(for functions) by the standard embedding-retraction method.
Choose a proper holomorphic embedding $\psi\colon \omega \to\C^N$
onto a closed complex (Stein) submanifold $\Sigma=\psi(\omega)\subset \C^N$,
and let $\eta \colon W\to \Sigma$ be a holomorphic retraction 
of an open neighborhood $W\subset\mathbb{C}^N$ of $\Sigma$ onto $\Sigma$
(see \cite{lit21}). The map
\[
	G=\psi\circ F\colon\Omega \to G(\Omega)\subset\Sigma
\]
is a $\mathcal{C}^{r}$-diffeomorphism of $\Omega$ onto its image in $\Sigma$,
and it is holomorphic in $\Omega_0$. Therefore 
$\overline\partial G$ and its derivatives 
$D^{\alpha}(\overline{\partial} G)$ of order 
$|\alpha|\leq r-1$ vanish on $\Omega_0$. In particular
we have
\begin{equation}
\label{estimates-G}
   \bigl|\overline{\partial}G(y)\bigr|=o\bigl(d(y,\Omega_0)^{r-1}\bigr),\qquad 
   \bigl|D^1(\overline{\partial}G)(y)\bigr|=o\bigl(d(y,\Omega_0)^{r-2}\bigr)
\end{equation}
as $y \to \Omega_0$, where $d(y,\Omega_0)$ denotes the distance from 
$y$ to $\Omega_0$ in $\theta$ (considered as a subset of $U_{\overline{D}}\times \C^k$).

By shrinking the neighborhood $U_{\overline{D}} \supset \overline{D}$, 
we may assume that there is a strictly plurisubharmonic $\mathcal{C}^r$-function 
$\rho \colon U_{\overline{D}} \to \mathbb{R}$ such that 
$D=\{z\in U_{\overline{D}} \mid \rho(z)<0\}$ and $d \rho(z)\neq 0$ for every 
$z\in bD=\{\rho=0\}$. For constants $\epsilon \ge 0$ and $M>0$ 
we set $\widetilde{\rho}(z,t)=\rho(z)+M|t|^2$ and define 
\begin{equation}
\label{ue}
	U_{\epsilon}=\{(z,t)\in \theta \mid z\in U_{\overline{D}},\ 
	\widetilde{\rho}(z,t)<\epsilon\}
\end{equation}
(see Figure \ref{slika6}). 
Clearly $\overline{D}\subset U_{\epsilon}$
for every $\epsilon>0$. By choosing $M>0$ sufficiently large we insure that $U_{\epsilon}\Subset \Omega$ for every sufficiently small 
$\epsilon\ge 0$. For such $M$ and for $\epsilon\ge 0$ small enough, 
$U_{\epsilon}$ is a strongly pseudoconvex domain with  
$\mathcal{C}^r$-boundary.

\begin{figure}[ht]
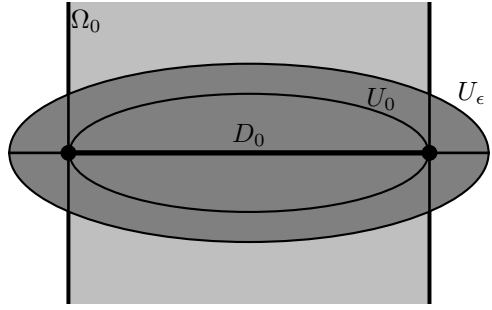

\psset{unit=0.8 cm}

\pspicture*(2,-0.5)(10,4.5)

\pspolygon[fillstyle=solid, fillcolor=lightgray, linewidth=1.5pt](3,-1)(9,-1)(9,5)(3,5)
\psellipse[fillstyle=solid,fillcolor=gray](6,2)(4,1.5)
\psline[linewidth=1pt](3,1)(3,4)
\psline[linewidth=1pt](9,1)(9,4)
\psellipse(6,2)(3,1)
\psline[linewidth=1pt,linearc=.15]{-}(1,2)(11,2)
\psline[linewidth=2pt,linearc=.15]{*-*}(3,2)(9,2)
\rput(3.3,4.2){$\Omega_0$}
\rput(9.7,3){$U_{\epsilon}$}
\rput(8.2,2.9){$U_0$}
\rput(6,2.3){$D_0$}

\endpspicture
\caption{The sets $U_{\epsilon}$}
\label{slika6}
\end{figure}

From the definition of $U_\epsilon$ it is clear that the distance from any point of 
$U_{\epsilon}$ to the set $\Omega_0=\pi^{-1}(\overline{D})\cap \Omega$ 
is of size $O(\epsilon)$. It then follows from (\ref{estimates-G}) that 
\begin{equation}\label{formula2}
	\|\overline{\partial} G \|_{L^{\infty}(U_{\epsilon})} = o(\epsilon^{r-1}) \textrm{  as  }\epsilon\to 0.
\end{equation}
There is a constant $C>0$ such that for every sufficiently small $\epsilon>0$ the equation $\overline{\partial}u=\overline{\partial}G$ has a solution $u_{\epsilon}$ on 
$U_{\epsilon}$, satisfying the uniform estimate 
\begin{eqnarray}
\label{formula1}
    \|u_{\epsilon}\|_{L^{\infty}(U_{\epsilon})} \leq C\, \|\dibar G \|_{L^{\infty}(U_{\epsilon})}
    =o(\epsilon^{r-1}).
\end{eqnarray}
The constant $C$ can be chosen independent of $\epsilon$, since
$U_\epsilon$ is ${\mathcal C}^2$-close to $U_0$
(see e.g.\ \cite[VIII/Theorem 8.5]{lit23} and also \cite{lit14,lit16,lit15}).
The map 
\[
	G_{\epsilon}= \eta \circ(G-u_{\epsilon})  \colon U_\epsilon \to \Sigma
\]
is hence holomorphic on $U_\epsilon$.

To complete the proof, we show that for every small $\epsilon >0$ the map 
$G_\epsilon \colon U_{\frac{\epsilon}{2}} \to \Sigma$ is actually biholomorphic 
onto its image that contains $\psi(A)$. The map
$F_\epsilon =\psi^{-1} \circ G_\epsilon$ is then biholomorphic
from $U_{\frac{\epsilon}{2}}$ onto a neighborhood of $A$ in $X$.

We use Lemma \ref{tl} locally on the set $U_{\frac{\epsilon}{2}}$ to obtain  
global estimates on the derivatives of $u_{\epsilon}$. 
In particular, we claim that 
\begin{equation}
\label{D1}
\|D^1u_{\epsilon}\|_{L^{\infty}(U_{\frac{\epsilon}{2}})}=o(1) \textrm{  as  } \epsilon \to 0.
\end{equation}
To see this, we fix a small $\epsilon_0 >0$ and choose finite open covers $\{W_j\}_{j=1} ^m$ and $\{V_j\}_{j=1} ^m$ of $\overline{U}_{\epsilon_0}$, satisfying $W_j\Subset V_j$ for every $j$. 
We may assume that $V_j\Subset U_j$ for every $j\in\{1,\ldots,m\}$, 
where $U_j$ is a local coordinate patch in $\theta$, and let $\varphi_j\colon U_j\to \varphi_j(U_j)\subset\mathbb{C}^{\dim \theta}$ be its corresponding biholomorphic map. 
Next, choose a constant $b>0$ such that
\begin{equation}
\label{trikotnik}
\bigl|\bigl(\widetilde{\rho}\circ\varphi_j ^{-1}\bigr)(x)-\bigl(\widetilde{\rho}\circ\varphi_j ^{-1}\bigr)(x')\bigr|<b |x-x'|  
\end{equation}
holds for all $x,x'\in V_j'=\varphi_j(V_j)$.
By the uniform continuity of $\varphi_j^{-1}$ on $\varphi_j ^{-1}(\overline{V}_j)$ 
there is also a constant $a>0$ such that 
$\varphi_j ^{-1}\bigl(B(w,a)\bigr)\subset V_j$ 
for every $w\in W_j'=\varphi_j (W_j)$. 
(Here $B(w,a)$ denotes a ball centered at $w$ and with radius $a$.) 
For any $\epsilon$ with $0<\epsilon< \min \{\epsilon_0,2ba\}$ and $w\in W_j '$ we set 
\[
	K_{w,j,\epsilon} = \varphi_j ^{-1}\left(\overline{B}\bigl(w,\frac{\epsilon}{2b}\bigr)\right) \subset V_j.
\]
By (\ref{trikotnik}) and the definition of $K_{w,j,\epsilon}$ we obtain
\[
  \bigl|\widetilde{\rho}(x)-\widetilde{\rho}\bigl(\varphi_j^{-1}(w)\bigr)\bigr|  =        
  \bigl|\bigl(\widetilde{\rho}\circ\varphi_j^{-1}\bigr) \bigl(\varphi_j(w)\bigr)-\bigl(\widetilde{\rho}\circ\varphi_j^{-1}\bigr)(w)\bigr|
  < b|\varphi_j(x)-w|<\frac{\epsilon}{2}
\]
for any pair of points $w\in\varphi_j(\overline{U}_{\frac{\epsilon}{2}}\cap V_j)$ and 
$x\in K_{w,j,\epsilon}$. 
Since $\varphi_j ^{-1}(w)\in U_{\frac{\epsilon}{2}}$, it follows from the above inequality that $\widetilde{\rho}(x)<\epsilon$ and hence $K_{w,j,\epsilon}\subset U_{\epsilon}$ for all $w\in\varphi_j(\overline{U}_{\frac{\epsilon}{2}}\cap V_j)$.
Using Lemma \ref{tl} for the function $u_{\epsilon}\circ \varphi_j^{-1}$ on $B(w,\frac{\epsilon}{2b})$, with $w\in \varphi_j(\overline{U}_{\frac{\epsilon}{2}}\cap V_j)$,
we get the following estimates for $|\alpha|\leq r-1$:  
\begin{eqnarray*}
  c\, \bigl|D^{\alpha}\bigl(u_{\epsilon}\circ\varphi_j^{-1}\bigr)\bigl(\varphi_j(w)\bigr)\bigr| &\le& 
  \frac{\epsilon}{2b}\bigl\| D^{\alpha}\bar{\partial} (u_{\epsilon}\circ\varphi_j^{-1})\bigr\|_{L^{\infty}(B(w,\frac{\epsilon}{2b}))} \\
  && +\bigl(\frac{\epsilon}{2b}\bigr)^{-|\alpha|} \,
   \bigl\|u_{\epsilon}\circ\varphi_j^{-1}\bigr\|_{L^{\infty}(B(w,\frac{\epsilon}{2b}))}  \\
\end{eqnarray*}
\[
\leq  \frac{\epsilon}{2b}\bigl\|D^{\alpha}\bar{\partial}
 		(u_{\epsilon}\circ\varphi_j^{-1})\bigr\|_{L^\infty(\varphi_j(V_j))} 	
 	+ \bigl(\frac{\epsilon}{2b}\bigr)^{-|\alpha|}||u_{\epsilon}||_{L^\infty(U_{\epsilon})}, 
\]
where we may assume that the constant $c>0$ is independent of the choice of $w\in \varphi_j(\overline{U}_{\frac{\epsilon}{2}}\cap V_j)$. As $||D^{\alpha}\bar{\partial}(u_{\epsilon}\circ\varphi_j^{-1})||_{L^{\infty}(\varphi_j(V_j))}$ 
is bounded from above by some positive constant, we further obtain  
\[
	\bigl\|D^{\alpha}(u_{\epsilon} \circ \varphi_j^{-1})\bigl\|_{L^{\infty}\bigl(\varphi_j(U_{\frac{\epsilon}{2}}\cap V_j)\bigr)} 
	\leq c'\epsilon +c''\epsilon^{-|\alpha|}||u_{\epsilon}||_{L^{\infty}(U_{\epsilon})},
\]
where $c'$ and $c''$ are some positive constants. 
Since the sets $\{V_j\}_{j=1}^m$ cover $U_{\frac{\epsilon}{2}}$, it then follows that
\begin{equation}\label{odvo2}
   ||D^{\alpha}u_{\epsilon}||_{L^{\infty}(U_{\frac{\epsilon}{2}})}\leq 
   \widetilde{c}\bigl( \epsilon  + \epsilon^{-|\alpha|}o(\epsilon^{r-1}) \bigr),\quad |\alpha|\leq r-1,
\end{equation}
for some positive constant $\widetilde{c}$. Clearly this implies (\ref{D1}).

The estimate (\ref{formula1}) shows that for $\epsilon>0$ small enough,
the image of the map $G-su_{\epsilon}\colon U_\epsilon\to \C^N$ 
$(s\in[0,1])$ is contained in some compact subset $W'\Subset W$ 
that can be chosen to be independent of $s$ and $\epsilon$.
Hence the family
\[
	G_{\epsilon, s}:= \eta \circ(G-su_{\epsilon})\colon U_{\epsilon}\to \Sigma,
   \qquad s\in [0,1]
\]
is a homotopy between $G|_{U_{\epsilon}}$ and the holomorphic map 
$G_{\epsilon}=G_{\epsilon,1}= \eta \circ(G-u_{\epsilon})$.

Since $G$ is a diffeomorphism, it follows from (\ref{formula1}) and (\ref{D1}) that $G_{\epsilon,s}$ is 
a diffeomorphism on $U_{\frac{\epsilon}{2}}$, 
provided that $\epsilon>0$ is small enough.
Furthermore, (\ref{formula1}) implies
\[
 \frac{\|G_{\epsilon,s}-G\|_{L^{\infty}(U_{\epsilon})}}{\epsilon}\to 0 
 \ \textrm{\ as\  } \epsilon\to 0  \ \textrm{\ uniformly\ in\ } s\in [0,1].
\]
It follows that
\begin{equation}
\label{boundary}
	G_{\epsilon,s}(bU_{\frac{\epsilon}{2}}) \in\Sigma\backslash G(\overline{U}_0)
\end{equation}
for every small $\epsilon>0$ and $s\in[0,1]$.

We now use a fact from topological degree theory (see the first two chapters in \cite{lit28}). 
It says that the degrees of two homotopic $\mathcal{C}^1$-mappings are equal, if they are 
calculated at a point that is not contained in the image of the boundary of the domain 
with respect to the homotopy. Since $G$ is a diffeomorphism, the degree of $G$ 
at a point $x\in G(\overline{U}_0)$ with respect to $U_{\frac{\epsilon}{2}}$ 
is equal to one. From (\ref{boundary}) it then follows that the degree of the homotopic map 
$G_{\epsilon}$ at a point $x\in G(\overline{U}_0)$ also equals one, and hence
$G_{\epsilon}(U_{\frac{\epsilon}{2}})$ contains the set $G(\overline{U}_{0})\supset A$.
Therefore the set 
\[
           \omega' = \psi^{-1}(G_{\epsilon}(U_{\frac{\epsilon}{2}}))\subset U \subset X
\] 
is an open Stein neighborhood of $A$ that is biholomorphic onto the domain 
$U_{\frac{\epsilon}{2}}\subset \theta$ via the biholomorphism 
$\Phi= F_{\epsilon}=\psi^{-1}\circ G_{\epsilon}$.

Finally, if $\epsilon >0$ is small enough then the preimage 
$A_\epsilon = F^{-1}_\epsilon (A) \subset \theta$ is 
sufficiently $\mathcal{C}^1$-close to $\overline D$ (the zero section
of $\theta|_{\overline D}$), so that $\pi$ projects it biholomorphically 
onto a compact domain $\overline{D}_\epsilon \subset U_{\overline{D}}$ with the interior $D'=D_{\epsilon}$. 
Since $A_\epsilon$ is biholomorphic to $A$, it has strongly pseudoconvex
$\mathcal{C}^r$-boundary, and hence the same is true for the domain $D_\epsilon$.
It follows by the implicit function theorem that
\begin{equation}
\label{A-eps}
	A_\epsilon = \phi_\epsilon (\overline{D}_{\epsilon}),
\end{equation}
where $\phi_\epsilon$ is a section of class $\mathcal{A}^r(D_\epsilon)$ 
of the restricted vector bundle $\theta|_{\overline{D}}$.

This completes the proof of Theorem \ref{izrek3}.
\end{proof}

\section{Approximation of mappings}
\label{approximation}

In this section we prove the following approximation theorem that
clearly includes Corollary \ref{corollary1}.

\begin{theorem}
\label{CR-approx}
Let $D$ be a relatively compact strongly pseudoconvex domain with $\mathcal{C}^l$-boundary 
$(l\geq 2)$ in a Stein manifold $S$, and let $f\colon \overline D \hookrightarrow X$
be an embedding of class $\mathcal{A}^l(D,X)$ onto $A=f(\overline{D})$. 
Then every $\mathcal{C}^r$-map $g\colon A\to Y$ $(r\in\{0,1,\ldots,l\})$ to a complex manifold $Y$, 
that is holomorphic in the interior $A\backslash bA$,  
is a $\mathcal{C}^r(A,Y)$-limit of a sequence of maps $g_j\colon U_j\to Y$ 
which are holomorphic in open neighborhods $U_j$ of $A$ in $X$. 
\end{theorem}

In general the domains $U_j$ of the approximating maps shrink down to $A$.
If $Y$ enjoys the {\em convex approximation property} (see \cite{FF:CAP})
then the Runge approximation theorem holds for maps of Stein manifolds 
to $Y$, and in this case we can choose the approximating maps $g_j$ 
on a fixed neighborhood of $A$ in $X$. 
This holds for example if $Y$ is a complex homogeneous manifold.

We begin by a brief review of the approximation results for maps in 
$\mathcal{A}^r(D,Y)$, where $D$ is a relatively compact strongly 
pseudoconvex domain in a Stein manifold. 
When $D\Subset\mathbb{C}^n$ and $Y=\mathbb{C}$, 
the uniform approximation of functions in 
$\mathcal{A}^0(D,\mathbb{C})$ by holomorphic function in some neighborhood of 
$\overline{D}$ is obtained by using integral kernel representation 
(see \cite{lit17}, \cite[Theorem 2.9.2]{lit6}, \cite{lit18}); 
for other proofs see also \cite{lit14, lit13}). 
The approximation results for functions in $\mathcal{A}^r(D)$ with 
$0\leq r\leq l$ and $2\leq l$ are obtained by using solutions of 
the $\bar\partial$-equation with $\mathcal{C}^r$-estimates 
(see \cite{Alt,siu2} for the case of $(0,1)$-forms and \cite{lit12} 
for the case of $(0,q)$-forms; 
see also \cite[VIII/Theorem 3.43]{lit23}, \cite{lit16}, \cite{lit15}). 
For approximation on weakly pseudoconvex domain see \cite{lit26,lit27}.

For compact sets that allow sufficiently regular Stein neighborhood bases
(e.g., the so-called {\em uniformly $H$-convex} sets)
one obtains holomorphic approximation theorems 
by using H\"ormander's $L^2$-method for solving the 
$\overline\partial$-equation, together with the interior
elliptic estimates, similar to those in Lemma \ref{tl} above.
For this approach see Chirka \cite{Cir69}.

These approximation theorems for functions easily imply 
the corresponding approximation theorems for maps 
$f\colon \overline D \to Y$ to more general complex manifolds $Y$, 
provided that the image $f(\overline D)$ admits a Stein neighborhood in $Y$.
(Just embed this Stein neighborhood into a Euclidean space and use a holomorphic
retraction onto the embedded submanifold.) Without a Stein neighborhood
the problem is more involved. In \cite{lit11} the authors proved 
by the method of holomorphic sprays that every map in $\mathcal{A}^r(D,Y)$ 
for $r \in\mathbb{Z}_+$ can be approximated in the $\mathcal{C}^r(\overline{D},Y)$-topology 
by maps that are holomorphic in open neighborhoods of $\overline D$ in 
the ambient Stein manifold $S$. (For a simpler proof when 
$r\geq 2$ see \cite[Theorem 2.6]{lit2}.)

\begin{proof}[Proof of Theorem \ref{CR-approx}]
By Theorem \ref{izrek3} it suffices to consider the following situation:
$D\Subset S$ is a strongly pseudoconvex domain with $\mathcal{C}^l$-boundary
in a Stein manifold $S$, $\pi\colon X\to S$ is a holomorphic vector bundle,
$\phi\colon \overline D\to X$ is a section of class $\mathcal{A}^r(D,X)$,
and $A=\phi(\overline D)$.

Let $U_0 \subset X$ be a strongly pseudoconvex domain of the 
form (\ref{ue}) (for $\epsilon=0$). Note that $U_0$ contains the 
zero section of $X|_D$, it is contained in $\pi^{-1}(D)$, 
and its boundary $bU_0$ is of class $\mathcal{C}^l$.
Let us write the section $\phi$ in the form $\phi(z)=\bigl(z,\varphi(z)\bigr)$,
where $\varphi(z)\in X_z=\pi^{-1}(z)$ denotes the fiber component. Set
\[
		U = \{ \bigl(z,t+\varphi(z)\bigr)\in X \mid (z,t)\in U_0\}.
\]
The set $U$ is obtained by translating $U_0$ in the fiber direction by
the section $\phi$. It is immediate that $U$ is strongly pseudoconvex
with $\mathcal{C}^r$-boundary and $\overline U \subset \pi^{-1}(\overline D)$, since we can choose $\phi$ sufficiently $\mathcal{C}^1$-close to the zero section.

Let $g\colon A\to Y$ be an $\mathcal{A}^r(A,Y)$-map to a complex manifold $Y$.
We extend $g$ to the map $G\colon \pi^{-1}(\overline D) \to Y$ by letting 
$G$ to be constant on each fiber $X_z$, $z\in\overline D$.
Then the restriction $G \colon \overline U \to Y$ is a map of class 
$\mathcal{A}^r(U,Y)$. By \cite[Theorem 1.2]{lit11} 
this map can be approximated in the $\mathcal{C}^r$-topology
by holomorphic maps from open  neighborhoods of $\overline U\supset A$ (in $X$)
to $Y$. This completes the proof.
\end{proof}

\smallskip
\textit{Acknowledgement.}
I wish to thank professor F. Forstneri\v{c} for  stimulating discussions and helpful remarks considering this paper, and esspecially for introducing me to the problem of the existence of Stein neighborhoods of embedded strongly pseudoconvex domains.

\bibliographystyle{amsplain}

\end{document}